\documentclass[onefignum,onetabnum]{siamart171218}
\usepackage{hyperref}
\usepackage{subdepth}
\usepackage{overpic}
\usepackage{amsmath,amstext,amsbsy,amssymb,mathdots}
\usepackage{booktabs,bm}
\usepackage{mathrsfs}
\usepackage{tikz,cite}
\usetikzlibrary{intersections}

\usepackage{courier}
\usetikzlibrary{positioning}
\usetikzlibrary{shapes,arrows}
\usepackage[fleqn,tbtags]{mathtools}
\usepackage{amsfonts}
\usepackage{relsize}
\usepackage{xcolor,colortbl}
\usepackage{psfrag}
\usepackage{alltt}
\usepackage{arydshln}
\usepackage{enumerate}
\usepackage{epstopdf}
\usepackage{enumitem}
\usepackage{multirow}
\usepackage{algorithmic}
\usepackage{framed}
\usepackage{cancel}
\usepackage{tikz}
\usepackage{tkz-euclide}
\usepackage{graphicx}
\usepackage{subcaption}
\DeclareCaptionSubType[Alph]{figure}
\newcommand*{\CopyCounter}[2]{%
  \expandafter\def\csname c@#2\endcsname{\csname c@#1\endcsname}%
  \expandafter\def\csname p@#2\endcsname{\csname p@#1\endcsname}%
  \expandafter\def\csname the#2\endcsname{\csname the#1\endcsname}}

\CopyCounter{Theorem}{ProposedProblem}
\CopyCounter{Theorem}{Proposition}
\CopyCounter{Theorem}{Property}
\CopyCounter{Theorem}{Claim}
\CopyCounter{Theorem}{Lemma}
\CopyCounter{Theorem}{Corollary}
\CopyCounter{Theorem}{Conjecture}
\CopyCounter{Theorem}{Definition}
\CopyCounter{Theorem}{Example}
\CopyCounter{Theorem}{Remark}
\CopyCounter{Theorem}{Question}
\CopyCounter{Theorem}{Condition}
\CopyCounter{Theorem}{Criterion}
\CopyCounter{Theorem}{Observation}
\theoremstyle{plain}

\newtheorem*{observ}[Observation]{Observation}

\newcommand{\domain}{\Omega}

\newcommand{\boundary}{\partial \domain}

\newcommand{\xbar}{\bm{\bar{x}}}

\newcommand{\Ctilde}{\tilde{C}}

\newcommand{\x}{\bm{x}}

\newcommand{\ba}{\bm{a}}
\newcommand{\adot}{\ba(\cdot)}

\newcommand{\F}{\mathcal{F}}

\newcommand{\y}{\bm{y}}

\newcommand{\J}{{\cal{J}}}
\newcommand{\X}{{\cal{X}}}

\newcommand{\R}{\mathbb{R}}
\newcommand{\E}{\mathbb{E}}

\newcommand{\ubar}{\bar{u}}
\newcommand{\A}{\mathcal{A}}
\newcommand{\D}{\mathcal{D}}
\newcommand{\Ubar}{\overline{U}}

\newcommand{\PA}{P^a}
\newcommand{\PG}{P^g}

\newcommand{\Pbar}{\bar{P}}
\newcommand{\PbarA}{\Pbar^a}
\newcommand{\PbarG}{\Pbar^g}
\newcommand{\Ap}{A_p}
\newcommand{\ApA}{\Ap^a}
\newcommand{\ApG}{\Ap^g}
\newcommand{\Vp}{V_p}
\newcommand{\VpA}{\Vp^a}
\newcommand{\VpG}{\Vp^g}

\newcommand{\pDomain}{\domain_p}
\newcommand{\pDomainA}{\pDomain^a}
\newcommand{\pDomainG}{\pDomain^g}

\newcommand{\pBoundary}{\boundary_p}
\newcommand{\pBoundaryA}{\pBoundary^a}
\newcommand{\pBoundaryG}{\pBoundary^g}

\newcommand{\Btilde}{b}
\newcommand{\Bmin}{\hat{B}}
\newcommand{\Bmax}{\check{B}}

\DeclareMathOperator*{\argmin}{arg\,min}

\colorlet{lightgray}{gray!40}

\newcolumntype{"}{@{\hskip\tabcolsep\vrule width 1pt\hskip\tabcolsep}}
\setlist[description]{font=\normalfont\space}

\ifpdf
  \DeclareGraphicsExtensions{.eps,.pdf,.png,.jpg}
\else
  \DeclareGraphicsExtensions{.eps}
\fi

\newsiamremark{remark}{Remark}
\newsiamremark{hypothesis}{Hypothesis}
\crefname{hypothesis}{Hypothesis}{Hypotheses}

\headers{Environmental Crime Modeling}{E. Cartee and A. Vladimirsky}

\title{Control-Theoretic Models of Environmental Crime\thanks{Submitted to the editors June 24th, 2019.
\funding{This work is supported in part by the National
Science Foundation grant DMS-1738010.  The second author's work is also supported by the Simons Foundation Fellowship.}}}
\author{Elliot Cartee\thanks{Department of Mathematics, Cornell University, Ithaca, NY 14853
  (\email{evc34@cornell.edu}).}
\and Alexander Vladimirsky\thanks{Department of Mathematics and Center for Applied Mathematics, Cornell University, Ithaca, NY 14853
  (\email{vladimirsky@cornell.edu}).}
}

\usepackage{amsopn}

\begin{document}
\maketitle
\begin{abstract}
We present two models of perpetrators' decision-making in extracting resources from a protected area.  It is assumed that the authorities conduct surveillance to counter the extraction activities, and that perpetrators choose their post-extraction paths to balance the time/hardship of travel against the expected losses from a possible detection.
In our first model, the authorities are assumed to use ground patrols and the protected resources are confiscated as soon as the extractor is observed with them.
The perpetrators' path-planning is modeled using the optimal control of randomly-terminated process. 
In our second model, the authorities use aerial patrols, with the apprehension of perpetrators and confiscation of resources delayed until their exit from the protected area.
In this case the path-planning is based on multi-objective dynamic programming.
Our efficient numerical methods are illustrated on several examples with complicated geometry and terrain of protected areas, non-uniform distribution of protected resources, and spatially non-uniform detection rates due to aerial or ground patrols.
\end{abstract}



\section{Introduction}

Mathematical modeling of environmental crimes in national parks 
has been a growing area of research.  Both the authorities and non-governmental conservation organizations 
strive to prevent
{\em illegal extraction activities}, including  
wildlife poaching and illegal logging. 
In 2010, Albers \cite{albers2010spatial} introduced a framework where the government chooses a patrol density (determining the probability of detection at each location) in such a way as to maximize the pristine area of the park.
The perpetrators then choose where to extract resources in such a way as to maximize their expected payoff. 
The problem of choosing patrol densities to maximize the pristine area has also been explored  in \cite{johnson2012patrol}.
In recent years,
patrol strategies based on graph-theoretic models (such as PAWS \cite{fang2017paws} and INTERCEPT \cite{kar2017cloudy}) have even been deployed in actual national parks.

The original Albers model \cite{albers2010spatial} was posed with many restrictive/unrealistic assumptions, including the radial symmetry of the protected area, patrol strategies, and resource distribution.
In a recent paper by D.J. Arnold et al \cite{arnold2018modeling}, these restrictions were largely removed, with numerical simulations on arbitrary two-dimensional domains and the terrain directly affecting the speed of perpetrators.   The key mathematical tool employed in \cite{arnold2018modeling} is the level-set method, developed to track the propagation of interfaces by solving time-dependent Hamilton-Jacobi PDEs \cite{OsherSethian_LSM}.
We begin by briefly reviewing both of these models in Section \ref{s:PriorWork}.

In this paper, we offer a significantly different perspective on how to generalize the Albers model to arbitrary domains.
As explained in Section \ref{s:ProblemDescription}, our main focus is on a careful treatment of the path-planning problem faced by the extractors.
Capturing their rational behavior yields a more accurate prediction of affected areas and allows for a better planning of optimal patrol strategies.
We show that the perpetrators optimization problem changes significantly based on the consequences of being spotted while transporting the protected resources.

In our first ``Model G'', where authorities use ``ground patrols'' to spot the extractors, any detection leads to an immediate confiscation, with perpetrators subsequently switching to the fastest path to leave the protected area.  The probability of being detected along any originally chosen path is not a sufficient description here, since it does not reflect the expected duration of that post-confiscation trip.  We handle this challenge by finding the optimal controls for a suitable {\em randomly-terminated process}.  In our second ``Model A'', where authorities run ``aerial patrols,'' it is assumed that the perpetrators remain oblivious of when a detection happens and are only apprehended later, once they reach the boundary of the protected area.  As a result, they always use their pre-selected trajectories, which are chosen to balance the probability of detection against the time or (cumulative difficulty) of the post-extraction trip.
In standard path-planning applications, the goal is usually to minimize the integral of some ``running cost'' along a chosen trajectory.
The key difficulty of this second model is that no such running cost can be explicitly defined to recover trajectories maximizing the extractor's expected profit.
Nevertheless, in section \ref{s:MultiObjective} we show that this problem can be solved using the tools of {\em multi-objective} optimal control theory.  

We discuss efficient numerical methods for 
the associated stationary Hamilton-Jacobi-Bellman PDEs in Section \ref{s:Numerics}. 
Our algorithms are illustrated  on a number of examples in Section \ref{s:Examples}, including on realistic terrain maps from Yosemite National Park in California and Kangaroo
Island in South Australia. 
We show that Models A and G often yield very different perpetrator trajectories and also 
differ in their predictions of the subdomain unaffected by extraction activities. 
We conclude by discussing possible extensions and directions for future work in section \ref{s:Conclusions}.

\section{Prior Work}
\label{s:PriorWork}

The model proposed by Albers in \cite{albers2010spatial} is based on a {\em one-period Stackelberg game} between the protected area managers (PAMs) and the extractors targeting a protected domain $\domain \subset \R^2.$  PAMs decide on their patrol strategy, defining the pointwise rate of extractor-detection $\psi:\domain \to \R_{+,0}$ and respecting the known budget constraints for patrolling efforts:
\begin{equation}
\label{eq:patrol_budget}
\int_{\Omega}\psi^{\gamma}(\x)d\x \le E,
\end{equation}  
with a constant $\gamma \geq 1$ modeling a possible increase in marginal costs of maintaining a higher detection rate.  

Once $\psi$ is chosen, it becomes known to extractors, who select which parts of $\domain$ to target, trying to maximize their expected profit (defined below).  Correspondingly, PAMs' goal is to select $\psi$ that will minimize the impact of the extractors' best/rational response.  This can be interpreted as either maximizing the ``pristine'' area unaffected by the extractors or maximizing the amount of non-extracted resources.  In Albers' original model both of these interpretations are equivalent since she assumes a homogeneous  distribution of protected resources.

Even more restrictively, the domain $\domain$ is assumed to be a unit disk, with extractors evenly distributed along its boundary.  This radial symmetry dramatically simplifies the problem and allows for analytical solutions for the optimal behavior of all participants.
E.g., it is easy to see that only radially symmetric $\psi(\x)$ would need to be considered -- Albers further restricts this to piecewise-constant  radial functions, with positive rate of detection in an annulus between radial distances $d_e$ and $d_i$ from $\boundary$
and $\psi=0$ everywhere else. 
The same radial symmetry ensures that extractors travel along straight lines from $\boundary$ toward the center of $\domain$ and their entire strategy can be encoded by the chosen depth of penetration $d \in [0,1].$
Albers assumes that extractors travel to their chosen location unimpeded (since they have not broken any law yet), then turn around and start extracting resources as they retrace their steps to $\boundary$.  If they are detected by PAMs on their way back, the extracted resources are confiscated and they have to return to $\boundary$ empty-handed, incurring the time/opportunity cost without any reward.
Thus, each extractor chooses $d$ to maximize their expected profit
\begin{equation}
\label{eq:Albers_profit}
P(d) \; = \; \left( 1 - \Psi(d) \right) B(d) - C(d) \, = \, B(d) - \left( \Psi(d)  B(d) + C(d) \right),
\end{equation}
where $\Psi, B,$ and $C$ are respectively the cumulative probability of being detected, the value of extracted resources, and the cost of traversing the path -- all corresponding to going distance $d$ into $\domain$ before starting to extract resources. 
While this is not explicitly specified in \cite{albers2010spatial}, $C$ can also include the cost of getting to the point from which the extraction starts.  A more subtle point is that $B$ includes the value of {\em all} resources extracted along the way back.
In the general setting, this continuous-along-the-path extraction would induce a competition among the extractors searching for the best paths through $\domain$, but the radial symmetry mitigates this effect.  However, this makes it harder to justify the {\em one-period} nature of the model: as the extractors' actions deplete the resources in parts of the protected area, that would obviously change the future optimizations problems (faced by both PAMs and the extractors) even in a radially symmetric setting. Such dynamic consequences of ``incomplete enforcement'' are acknowledged in \cite{albers2010spatial} and also well-known in practice; e.g., \cite{Robinson_2008a, Robinson_2008b}.

A recent paper by D.J. Arnold et al  in \cite{arnold2018modeling} extended the above general approach to more realistic protected areas.  Their model allows for general domain geometries, resource distributions, terrain effects, and PAMs' strategies.  One important distinction is that, unlike Albers, they model the extraction as occurring at isolated points.  An advantage of this approach is that this reduces competition among the extractors (since the paths-to-$\boundary$ traversed by others don't affect the value of what you are currently trying to extract).  
For an extractor targeting a location $\x \in \domain,$ they define the expected profit as 
\begin{equation}
\label{eq:UCLA_profit}
P(\x) \; = \; B(\x) - \tilde{C},
\end{equation}
where $B$ is the value of resources available for extraction in a small neighborhood of $\x$,
while $\tilde{C}$ is the effective cost of a post-extraction trajectory, accounting both for the time/difficulty of travel from $\x$ to $\boundary$ and for the expected losses from a possible capture.
A perpetrator chooses this trajectory to minimize $\tilde{C}$ (and thus maximize $P$).
In \cite{arnold2018modeling} this optimization is performed implicitly by 
using a level set method originally developed for tracking the propagation of interfaces \cite{OsherSethian_LSM}.
The authors show that, in an Albers-style radially symmetric setting with unit walking speed of perpetrators, their approach produces the same $P(\x)$ as predicted by formula 
\eqref{eq:Albers_profit}.  
A comparison with \eqref{eq:UCLA_profit} thus suggests an explicit interpretation $\Ctilde(\x) = \Psi B(\x) + C.$

For any chosen constant $\Btilde>0,$ Arnold et al compute $\Ctilde(\x_0)$ simultaneously for all extraction positions $\x_0$ such that $B(\x_0) = \Btilde.$ 
The idea is to find equal-effective-cost curves 
by using the level sets of an auxiliary time-dependent function $\phi(\x,t)$ satisfying
a Hamilton-Jacobi PDE
\begin{equation}
\label{eq:LSM}
\phi_t + F(\x,\Btilde) |\nabla \phi| \; = \; 0, \qquad t>0, \x \in\domain,
\end{equation}
with initial conditions $\phi(\x,0) = $ distance from $\x$ to $\boundary.$  Assuming that $\Gamma(t)$ is some interface starting from $\Gamma(0) = \boundary$ and monotonically advancing into $\domain$ with the normal speed $F$, its evolution is encoded by the zeroth level-set of $\phi$ solving the above PDE; i.e.,
$\phi(\x, t)  = 0 \; \Longleftrightarrow \; \x \in \Gamma(t).$
For the purposes of this model,
the authors choose their normal speed of the interface to be
\begin{equation}
\label{eq:UCLA_F}
F(\x, \Btilde) \; = \; \frac{1}{ 1/ f(\x) \, + \, \alpha \psi(\x) \Btilde},
\end{equation}
where $f(\x)$ is the extractor's local speed of motion defined by the terrain at $\x$, and a parameter $\alpha>0$ encodes the degree of extractor's risk-aversion.  This formula is motivated phenomenologically: $1/f(\x)$ represents the extractor's ``slowness of motion''
and  $\psi(\x) \Btilde$ represents his expected losses (per unit time, while traveling through $\x$).  An increase in either of these terms should make  the effective cost  higher.  Once $\phi$ is computed, the authors define $\Ctilde$ implicitly through
$$
\phi(\x_0, \Ctilde(\x_0)) \; = \; 0, \qquad \forall \x_0 \text{ such that } B(\x_0) = \Btilde, 
$$
with an optimal path to $\boundary$ from any such $\x_0$ recovered via gradient descent in $\phi$.
The process is then repeated for multiple values of $\Btilde,$ until $\Ctilde$ (and thus also $P$) are defined for all gridpoints of some grid imposed on $\domain.$

In another departure from Albers' original model, 
the authors assume that extractors will target all locations where $P$ is within $(1-\varepsilon)$ factor from its maximal value; i.e., the extraction happens on the set 
$$
\domain_e \; = \; 
\left\{ 
\xbar  \in \domain \, \mid \, P(\xbar) \, \geq  \, (1-\varepsilon) \max\limits_{\x \in \domain} P(\x) 
\right\}.
$$ 
(One drawback of this assumption is that extraction will always happen on some part of $\domain$, regardless of how high the detection rate $\phi(\x)$ might be.)
The authors then define the pristine area by excluding not only the above but also the union of narrow corridors around optimal trajectories from $\domain_e$.  The latter are approximated by starting from a finite number of points selected uniformly at random in $\domain_e$.
Since there is no extraction along the path, the motivation for this exclusion 
might perhaps be based on extraction-unrelated damage to the protected area. 
While \cite{arnold2018modeling} does not present a systematic method for optimizing PAMs' choice of $\psi$, the authors use the above interpretation to compare the effectiveness of several (ad hoc) patrol strategies staying within the specified patrol budget.  
Their simulations (using $\gamma = 1, \alpha = 1,$ and $\varepsilon = 0.085$) are based on the terrain data from Yosemite National Park in California and Kangaroo Island in South Australia, which they have also kindly shared with us to conduct the numerical experiments presented in section \ref{s:Examples}. 

We end this section by discussing the difference between path planning to avoid detection vs capture/interception.  Even if the former always leads to the latter, there is still a very subtle modeling question of whether the capture happens immediately upon detection and, if there is a significant delay, whether the extractors learn that they have been detected before the capture.  These issues are not important in \cite{albers2010spatial} since the extractors' paths are always radial in Albers' setting.  But Arnold et al describe path-planning that strikes a balance between the time of travel and risk of detection all the way to the boundary --  with no provision for switching to quicker/easier paths if authorities manage to intercept the extractors en route and confiscate their haul. Thus, we believe that the model in \cite{arnold2018modeling} is more suitable for the setting where the extractors remain oblivious of their detection (perhaps in the case of aerial surveillance) and are only intercepted upon reaching $\boundary$.  

\section{Extractor's path-planning}
\label{s:ProblemDescription}

This section defines the optimization problem faced by a perpetrator 
deciding whether to extract resources from a fixed location $\x \in \domain$.
As in \cite{arnold2018modeling}, our approach is based on ``pointwise'' extraction decisions, but since these resources are typically continuously distributed in the protected area $\domain,$ this requires further clarification.  We will consider a fine grid $\X$ of possible extraction sites, with each $\x \in \X$ representing a center of its extraction cell $N_{\x},$ and $B(\x)$ representing the value of all resources that can be extracted from 
that cell. 
The same grid $\X$ will be also used to solve the PDEs numerically in section \ref{s:Numerics}.\\

\noindent
We will further assume that
\begin{enumerate}
\item
The extractor can enter 
$\domain$ through any part of $\boundary$ and has no reason to avoid surveillance from PAMs until he reaches the extraction site $\x.$
\item
While traveling through $\domain$, the extractor may freely choose any measurable ``control function'' $\ba: \R \to S^1,$ specifying his chosen direction of motion for all times $t$.  This defines the corresponding trajectory since his (isotropic) speed of motion $f:\domain \to \R_+$ 
reflects the local terrain and is a part of the input data.
\item 
The time/difficulty cost of traveling through $\domain$ is reflected by integrating some known running cost $K:\domain \to \R_+$ along his chosen trajectory.  One natural choice is $K(\x)=\kappa,$ reflecting the monetary value the extractor associates with each hour of his time spent on a trip. (To normalize the units, all of our numerical examples use $\kappa = 1$ for all pre-extraction and post-extraction travel through $\domain.$)  But a non-constant $K$ can be also used to reflect the extractor's distaste for some parts of the protected area (e.g., due to vegetation). 
\item
Once he reaches $\x$,  the extraction in $N_{\x}$ is instantaneous and the extractor immediately starts along his (possibly different) path to the boundary.
\item
On the way back to $\boundary$, he is trying to minimize both the time/cost of travel and the probability of being spotted  by PAMs through aerial surveillance.
\item
The extractor has full prior knowledge of PAMs' location-dependent detection rate $\psi: \domain \to \R_{+,0}.$
\item
Rational extractors will target all sites $\x \in \X$ for which the expected profit $P(\x)$ of their optimal roundtrip is above some threshold level $\tilde{p}$.  
This leaves a pristine area
$$
\pDomain \; = \; \{ \x \, \mid \, P(\x) \leq \tilde{p} \}.
$$  
(Our numerical experiments use $\tilde{p}=0.$) 
Note that this is different from the approach in \cite{arnold2018modeling} not only in providing an absolute threshold, but also because we don't exclude from  $\pDomain$ any corridors 
around post-extraction trajectories.  We believe that this is more consistent with the notion of 
localized extractions. 
\end{enumerate}

Computing the expected profit $P(\x)$ requires solving two different optimization problems:
(a) finding the best (PAMs-ignoring) pre-extraction trajectory from $\boundary$ to $\x$ and then (b) finding the best post-extraction trajectory from $\x$ to $\boundary$, balancing the cost of travel and risk of detection.  We will start by addressing (a) 
in section \ref{ss:pre-extract}
since it is easy to solve by standard tools of {\em single-criterion deterministic optimal control theory}.
Below we provide only a very brief introduction, referring to standard references (e.g., \cite{bressan2007introduction} and \cite{bardicapuzzodolcetta}) for technical details.

We will then use two models to address (b) under very different assumptions about the nature of PAMs' enforcement activities and the information available to perpetrators: \\

\noindent
{\bf Model G}  (ground patrols): \\
If the extractor is detected while traveling with protected resources, he is immediately apprehended and the resources are confiscated.  From there on, the perpetrator does not attempt extracting anything on this trip and simply seeks to minimize the integral of $K$ on his way to $\boundary.$\\

\noindent
{\bf Model A}  (aerial patrols): \\
The extractor has no way of learning if or when the detection occurs.  But if it happens, he is later apprehended with probability one upon reaching $\boundary$ and his entire haul is confiscated.\\

We will show that these two interpretation require fairly different mathematical tools:  Model G is best recast as an optimal control of a randomly-terminated process (subsection \ref{ss:post-extraction_G}), while Model A requires the techniques from multi-objective dynamic programming (subsection \ref{ss:post-extraction_A}).  The numerical methods for both versions are covered in section \ref{s:Numerics}.

\subsection{Getting to the extraction point}
\label{ss:pre-extract}
We will use $R(\x)$ to denote the cost of the optimal pre-extraction trajectory from 
$\boundary$ to $\x$. 
But to make the discussion closer to the canonical optimal control setting, we will 
reverse the direction of the trajectory in (a), looking instead for a $K$-optimal way to reach 
$\boundary$ from $\x$. (Based on our isotropy assumptions, this direction reversal will change neither the time of travel nor the integral of $K$ along any considered trajectory.)  The added benefit is that the same discussion will be a useful starting point for solving (b).

The extractor's time-dependent position in the domain will be specified by 
\begin{equation}\label{eq:dynamics}
\y'(t) \; = \; f\left( \y(t) \right) \ba(t), \qquad \y(0) \; = \; \x,
\end{equation}
Once  $\ba(\cdot)$ is chosen, this defines the exit time $T_{\x,\ba(\cdot)} \; =  \; \min \{t \ge 0 \, | \, \y(t) \in \boundary \}$ and the cumulative cost of starting at $\x$ and using $\ba(\cdot):$
$$
\J \left( \x, \ba(\cdot) \right) \; = \; \int_0^{T_{\x,\ba(\cdot)}}K(\y(t)) \, dt.
$$
If the extractor only cares about selecting a path to minimize this $\J,$ it is easy to accomplish by traditional tools of single-criterion dynamic programming introduced by Richard Bellman in the 1950s. 

The {\em value function} $u(\x)$ is defined as the minimum cost one has to pay starting from $\x$; i.e.,
$u(\x) = \inf_{\ba(\cdot)} \J \left( \x, \ba(\cdot) \right).$ 
A Taylor series expansion along the optimal trajectory can be used to derive
a Hamilton-Jacobi-Bellman (HJB) PDE that $u$ must satisfy if it is sufficiently smooth.  For our isotropic cost and dynamics, this PDE is merely an Eikonal equation
\begin{equation}
\label{eq:Eikonal}
|\nabla u (\x) | f(\x) \; = \; K(\x)    
\end{equation}
solved with Dirichlet boundary conditions $u=0$ on $\boundary$.
Unfortunately, in general \eqref{eq:Eikonal} does not have a classical/smooth solution, and there are infinitely many weak (Lipschitz-continuous) solutions. Additional test conditions were introduced by Crandall and Lions \cite{crandall1983viscosity} to  select among them the unique {\em viscosity solution} coinciding with the value function of the above control problem.  
Several highly efficient  methods for solving Eikonal PDEs numerically were developed in the past 25 years; see a brief overview at the end of section \ref{s:Numerics}. 

We note that taking $K=1$ results in $J \left( \x, \ba(\cdot) \right) = T_{\x,\ba(\cdot)}.$
So, both the distance-from-boundary $d(\x)$ and the min-time-from-boundary $\tau(\x)$ can be also found by solving the Eikonal equations
$|\nabla d|=1$ and $|\nabla \tau| f = 1$ with zero boundary conditions. If $K=\kappa$, then the minimum cost of reaching $\x$ from $\boundary$ is 
simply
$$
R(\x) = \kappa \tau(\x).
$$ 

The characteristic curves of HJB equations are the optimal trajectories, which in the Eikonal case coincide with the gradient lines of the viscosity solution.  I.e., once the value function is computed, an optimal trajectory can be found via gradient descent in $u$.  Since this function is differentiable almost everywhere, the optimal trajectories are unique for almost all $\x \in \domain.$

\begin{remark}
We note that the function $\Ctilde$ used in \cite{arnold2018modeling} can also be found directly (without using a level-set formulation) by solving an Eikonal equation $|\nabla \Ctilde| F = 1.$  
While the level set method is a much more general interface tracking approach, the fast Eikonal solvers are likely to be more efficient for the current application.
\end{remark}

\subsection{Returning with the loot (Model G: ground patrols)}
\label{ss:post-extraction_G}
Here we assume that a detection leads to an immediate apprehension and confiscation of extracted resources.  Thus, it is natural to split the post-extraction trajectory into two parts:
a pre-detection trajectory (leading all the way to $\boundary$ if the detection was avoided) and a post-detection trajectory (chosen by a perpetrator to simply minimize the integral of $K$ up to $\boundary$ after the confiscation).  
Suppose the perpetrator starts from an extraction point $\x_0$ and manages to reach some $\x \in \domain$ undetected.  From there on, he uses a control $\adot,$ moving from $\x = \y(0)$ along a trajectory $\y(t)$ with a possible detection at some time $T$ before reaching the boundary.  His remaining cost is
$$
\J(\x, \adot, T) \; = \; 
\begin{cases}
\int_0^{T} K(\y(t)) \, dt \, + \, B(\x_0) \, + \, R \left(\y(T) \right),  
& \text{ if } T <  T_{\x,\ba(\cdot)} \\
& \text{(i.e., if detected)}\\ \\
\int_0^{T_{\x,\ba(\cdot)}} K(\y(t)) \, dt,  & \text{ otherwise.}
\end{cases}
$$
Of course, the time of capture $T$ is random and its distribution depends on the chosen pre-detection trajectory.
So, the perpetrator's value function is defined to minimize the expected remaining cost (from $\x$ to $\boundary$):  
$$
\ubar(\x) \, = \, \inf_{\adot} \E_T^{} \left[  \J(\x, \adot, T) \right].
$$
We note that it also depends on the value of extracted resources, $\Btilde = B(\x_0)$, but we leave this implicit to simplify the notation.

To obtain the Bellman optimality condition, we assume that the detection will not happen for the next $s$ seconds, but might occur after that period with probability $s \psi(\x).$  
This yields
$$
\ubar(\x) \, = \, \inf_{\adot} \left\{  \int_0^{s} K(\y(t)) \, dt  \, + \, 
s \psi(\x) \left[ \Btilde + R \left(\y(s) \right) \right]
\, + \, 
\left( 1 - s \psi(\x) \right) \ubar(\y(s)) \right\} \, + \, o(s).
$$
Assuming that $\ubar$ is sufficiently smooth, a Taylor expansion of the above yields a Hamilton-Jacobi PDE
\begin{align}
\label{eq:ubar}
|\nabla \ubar(\x) | f(\x)  \; &= \;  K(\x)  + \psi(\x)  \left( \Btilde + R(\x) - \ubar(\x) \right), 
\qquad
& \text{on } \domain;\\ 
\nonumber
\ubar(\x) \; & = \; 0,
\qquad
& \text{on } \boundary. 
\end{align}
Since this PDE generally does not have a classical/smooth solution, that derivation is only formal, but viscosity solution theory \cite{crandall1983viscosity, bardicapuzzodolcetta} allows us to pick the weak solution coinciding with the value function. 
The above derivation is similar to what is used in optimal control of {\em randomly-terminated process}.    Such processes arise in
production/maintenance planning \cite{Boukas1990}, economic growth and global climate change modeling \cite{HaurieMoresino2006},
multi-generational games \cite{Haurie2005}, and optimizing the routing of emergency response vehicles \cite{AndrewsVlad}. 
In the current context, ``termination'' is either ``detection/capture'' (since we already know the cost of optimal actions after that point) or ``reaching $\boundary$ undetected''.  The latter possibility yields Dirichlet boundary conditions in \eqref{eq:ubar} and avoids the issues related to ``free boundary'' and quasi-variational inequalities considered in \cite{AndrewsVlad}.

Once $\ubar$ is found, the optimal pre-detection trajectories can be found by gradient descent.  The extractor's expected profit can be also computed as
$$
\PG(\x_0) \; = \; \Btilde - \ubar(\x_0) - R(\x_0),
\qquad 
\forall \x_0 \in X \text{ such that } B(\x_0) = \Btilde,
$$
where $R(\x_0)$ accounts for the cost of a pre-extraction trajectory.

We note that it is easy to find an upper bound on $\ubar$ by committing to use some $\adot$ up until $\boundary$ even if the detection occurs earlier.  
If that $\adot$ is selected to minimize the integral of $K,$  this yields  $\ubar(\x_0) \leq \Btilde + R(\x_0),$  but a sharper bound is obtained by choosing
a control optimal for Model A considered below.  This is why $\PG \geq \PA$ and $\pDomainG \subset \pDomainA.$

\subsection{Returning with the loot (Model A: aerial patrols)}
\label{ss:post-extraction_A}
On the way back to $\boundary$, the extractor needs to balance the cumulative cost of the path with expected losses resulting from a possible detection by PAMs.  
This is even more important under the current model, where the extractor does not learn whether he was detected until reaching $\boundary$. 
Here we formally derive the probability that an extractor is not detected while traversing a path $\y: [0,T] \to \domain.$
Suppose the probability of detection at a location $\x$ over a small time interval $ \Delta t$ is $\psi(\y) \Delta t$. 
Then the probability $\mathcal{P}_N$ of making it through $N$ such consecutive intervals of length $\Delta t=T/N$ without ever being detected is:
\begin{equation*}
\mathcal{P}_N = \left(1-\psi(\y_1)\Delta t \right)  \left(1- \psi(\y_2) \Delta t \right) \cdots  \left(1-\psi(\y_N) \Delta t \right) 
\; = \; \prod_{i=1}^N  \left(1-\psi(\y_i) \Delta t \right).
\end{equation*}
Taking the logarithm of both sides,
\begin{equation*}
\log(\mathcal{P}_N) = \sum_{i=1}^N  \log\left(1-\psi(\y_i) \Delta t\right) \; = \; -\sum_{i=1}^N \left( \psi(\y_i) \Delta t + \mathcal{O}\left(\Delta t^2\right)\right).
\end{equation*}
Taking the limit as $N \to \infty$, the probability of not being detected along this path until at least the time $T$ is 
\begin{equation}\label{eq:detection_probability}
\mathcal{P} = e^{-\int_0^T\psi(\y(t))dt}.
\end{equation}

Our emphasis on careful modeling of the cumulative probability of detection $\Psi = 1 - \mathcal{P}$ 
is perhaps the main distinction from earlier models. 
Both \cite{albers2010spatial} and \cite{arnold2018modeling} 
implicitly assume
$\Psi \approx \int_0^T\psi(\y(t))dt.$ 
As a first-order approximation of a concave function ($1- e^{-\xi} = \xi + O(\xi^2)$), this overestimates $\Psi$ for every trajectory.
It also yields a simplified optimization problem, but with significant limitations described in Remark \ref{rem:linearize}. 
Instead, we find the extractor's expected profit by using the accurate $\mathcal{P}$ to optimize his post-extraction trajectory:
\begin{align}\label{eq:multi_objective}
\begin{split}
&\y'(t) = f(\y)\ba(t), \qquad \y(0) = \x \\
&T_{\x,\ba(\cdot)} =  \min \{t \ge 0 \, | \, \y(t) \in \boundary \} \\
&\J_1(\x,\adot) = \int_0^{T_{\x,\ba(\cdot)}} \psi(\y(t))dt\\ 
&\J_2(\x,\adot) = \int_0^{T_{\x,\ba(\cdot)}} K(\y(t))dt\\
\PA(\x) = &\sup_{\ba(\cdot) \in \A} \left\{ B(\x)e^{-\J_1(\x,\adot)} - \J_2(\x,\adot)\right\} - R(\x),
\end{split}
\end{align}
where $\A$ is the set of \emph{admissible} controls, i.e., the set of measurable functions from $\R$ to $S^1,$ and
$R(\x)$ is the optimal cost of a pre-extraction path defined in subsection \ref{ss:pre-extract}.

\begin{remark}
\label{rem:linearize}
If $\J_1(\x_0,\adot) \ll 1$  for all $\x_0$ and all controls $\ba(\cdot)$ close to optimizing 
\eqref{eq:multi_objective},  a linear approximation of the detection probability $\Psi \approx \J_1$ is reasonable and leads to a significantly simpler optimization problem
similar to those considered in \cite{albers2010spatial} and \cite{arnold2018modeling}.  
In that case, 
\begin{equation}
\label{eq:linearizedPaApprox}
\PA(\x) \; \approx \; B(\x) - \inf_{\ba(\cdot) \in \A} \biggl\{ B(\x) \J_1(\x,\adot) + \J_2(\x,\adot)\biggr\} \, -  \, R(\x).
\end{equation}
Since the linearized approach exaggerates $\Psi$,
we know that this estimate is actually a lower bound on $\PA(\x).$
The minimizing control in \eqref{eq:linearizedPaApprox} can be found by 
solving the Eikonal PDE 
\begin{equation}
\label{eq:linearizingEikonal}
|\nabla u| f(\x) \; = \; B(\x_0) \psi(x) + K(\x)
\end{equation}
with zero boundary conditions.  
If one prefers the level set formulation, this is equivalent to using 
the interface speed of
$$
F(\x, \Btilde) = f(\x) \, /  \left( K(\x) + \Btilde \psi(\x) \right),
$$
which is similar to but not quite the same as the formula \eqref{eq:UCLA_F} 
used in \cite{arnold2018modeling} --
even in the case of risk neutral extractors (i.e., $\alpha = 1$)
concerned with their time along the trajectory (i.e., $K(\x)=1$).

Unfortunately, $\J_1$ can be arbitrarily large, which makes this ``$\Psi$-linearization'' unsuitable for most realistic situations.  
(Indeed, in our numerical examples of section \ref{s:Examples}, $\J_1 > 1$ is often observed along many optimal paths.)
This is the key reason why the more complicated multi-objective path-planning methods considered below are actually needed.  
\end{remark}

\begin{remark}
\label{rem:higher_dim}
The non-linear dependence on $\J_1$ is the reason why $\PA$ cannot be computed by standard tools of single-objective optimal control on $\domain.$  But this actually can be accomplished if we are willing to increase the dimension of our planning space.  We note that $\mathcal{P}$ satisfies an ODE $\mathcal{P}'(t) = -\psi(\y(t)) \mathcal{P}(t)$ along the chosen trajectory, with
$\mathcal{P}(0) = 1, \, \y(0) = \x_0.$  One could define a new value function $w(\x, p)$ as the minimum expected cost of the remaining path to $\boundary$ if the extractor has already moved from $\x_0$ to  $\x$ and the probability that he has not been detected so far is $p.$  A formal control-theoretic argument shows that $w$ should satisfy the following Hamilton-Jacobi PDE on $\domain \times (0,1]:$
\begin{equation}
\label{eq:xp_HJB}
\psi(\x) p \frac{\partial w}{\partial p} \, + \, | \nabla_{\x} w | f(\x) 
\; = \;  K(\x)
\end{equation}
with boundary conditions $w(\x,p) = (1-p) \Btilde$ on $\boundary \times (0,1].$
Once $w$ is computed,  the expected profit can be found as
$\PA(\x_0) = \Btilde - w(\x_0,1) -P_0(\x_0)$ for every $\x_0$ such that $B(\x_0) = \Btilde.$
The need to solve this PDE on a higher dimensional domain for many $\Btilde$ values
would make this approach computationally expensive, with additional numerical difficulties due to that fact that the coefficient $(\psi(\x) p)$ can be arbitrarily small or even zero.  This is why we opt instead to use multi-objective dynamic programming on the original planning space $\domain.$
\end{remark}

\section{Multi-Objective Approach (Model A)}
\label{s:MultiObjective}

For convenience, we will define $$\F(\J_1,\J_2) = Be^{-\J_1}-\J_2,$$ and in a slight abuse of notation, drop the first argument from both $\J_1$ and $\J_2$ (since $\x$ is fixed throughout this section), and also drop the argument from $\adot$.

Ideally, one would prefer to find a control minimizing both $\J_1$ and $\J_2$ simultaneously.  Somewhat surprisingly this is actually possible when $K$ is constant and PAMs use ``$\tau$-banded'' patrol densities; i.e., when $\psi$ is really a function of the min-time-from-$\boundary$
$\tau(\x)$ defined in section \ref{s:ProblemDescription}.  
It is not hard to show that in this case the trajectories recovered from the Eikonal \eqref{eq:Eikonal} will also minimize the probability of detection.
When $f$ is also constant, such trajectories will be simply straight lines (from the extraction point 
$\x$ to the closest point on $\boundary$), which will also minimize the probability of detection for all
``distance-to-$\boundary$-banded'' patrol strategies  $\psi(d)$ (including the homogeneous $\psi(\x) = \psi_0$) considered both in \cite{albers2010spatial} and \cite{arnold2018modeling}.  
This is also the setting where there is no difference between the predictions based on Model A and Model G -- even if the loot is confiscated, there is no incentive to switch to another trajectory.

For the above scenarios the extractor's path-planning is fairly trivial.  But in the general case, these two minimization criteria are in conflict, and one has to switch to a notion of {\em Pareto optimality}.

\subsection{Pareto Front}\label{ss:PF}

We will say that a control $\ba_1$ \emph{dominates} another control $\ba_2$ if $\J_1(\ba_1) \le \J_1(\ba_2)$ and $\J_2(\ba_1) \le \J_2(\ba_2)$ with at least one of those inequalities strict. 
All $\ba$'s not dominated by any other control are called {\em Pareto optimal}.  They correspond to trajectories that cannot be improved with respect to both criteria simultaneously.
It will be useful to consider the \emph{Pareto Front} 
$$
PF \; = \; \left\{
\left (\J_1(\ba), \J_2(\ba) \right)
\, \mid \, \ba \text{ is Pareto optimal } 
\right\}.
$$

In addition to control functions,  the extractor could in principle also consider mixed/probabilistic strategies, which specify a probability distribution over a set of available controls.  A strategy selecting any control with probability one is usually called pure or deterministic.  
Since $\F$ is decreasing in both arguments, it follows that $\F\left(\J_1(\ba_1),\J_2(\ba_1)\right) > \F\left(\J_1(\ba_2),\J_2(\ba_2)\right)$ for any control $\ba_1$ dominating $\ba_2$.
From this we see that any optimizer of \eqref{eq:multi_objective} over the set of pure strategies must belong to PF.
We also note that the extractor has no incentive to consider mixed strategies, since the expected payoff for using strategy $\ba_1$ with probability $\theta$ and $\ba_2$ with probability $1-\theta$ is
\begin{equation*}
\theta \F\bigl(\J_1(\ba_1),\J_2(\ba_1)\bigr) + (1-\theta)\F\bigl(\J_1(\ba_2),\J_2(\ba_2)\bigr) \, \le \, \max_{i=1,2} \F\bigl(\J_1(\ba_i),\J_2(\ba_i))
\end{equation*}
So in practice, we will compute $\PA(\x)$ by maximizing $\F(\J_1,\J_2)$ over the $(\J_1,\J_2) \in PF.$

There are a number of techniques for finding Pareto-optimal controls including {\em scalarization}  on $\domain$ \cite{mitchell2003continuous} and other methods relying on augmented PDEs on an expanded state space \cite{kumar2010efficient, DesillesZidani}.
To scalarize, one chooses some $\lambda \in [0,1]$ and optimizes a convex combination $\J^{\lambda} = \lambda \J_1 + (1-\lambda) \J_2$.  The resulting $\lambda$-optimal trajectory 
is also Pareto-optimal (see Figure \ref{fig:pf_scal}),  
and the procedure is applied repeatedly for different $\lambda$ values.
Unfortunately, such scalarization only finds points on the convex part of the Pareto Front \cite{das1997closer}.
This is a serious disadvantage since non-convexities of PF are both common and important in many control applications including robotic navigation \cite{kumar2010efficient, UTRC_2}.
In other cases, such as our recent work on surveillance-evasion games \cite{Static_SEG, TD_SEG}, special features of the problem might make scalarization sufficient.  Below we explain why this is also the case for
the post-extraction path-planning.

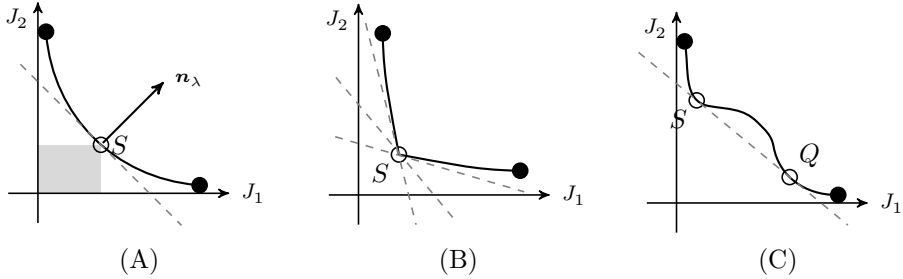
\begin{figure}[hhhh]
\begin{center}
\def\tikzPfScale{0.43}
\begin{tabular}{c c c}
\begin{tikzpicture}[
point/.style={draw,shape=circle,inner sep = 0mm,minimum size = .2cm},
scale = \tikzPfScale,
>=stealth',
shorten >=1pt,
shorten <=1pt,
auto,
node distance = 1.5cm,
semithick]

\def\maxx{5}
\def\maxy{5}

\coordinate (P) at (0.39*\maxx, 0.3*\maxy);

\fill [gray!30!white] (0,0) rectangle (P); 
\node[point] (Ppt) at (P) {};
\node[right]       (Plb) at (P) {$S$};

\coordinate (X1) at (0.05*\maxx, 1.00*\maxy);
\coordinate (Xn) at (1.00*\maxx, 0.05*\maxy);

\draw[thick]
	(X1) to[out=-85,in=175] 
	(Xn);
	
\node[point,fill=black] at (X1) {};
\node[point,fill=black] at (Xn) {};

\coordinate (H1) at (-1,0);
\coordinate (H2) at (\maxx+1,0);
\coordinate (V1) at (0,-1);
\coordinate (V2) at (0,\maxy+1);

\draw[->] (H1) -- (H2);
\draw[->] (V1) -- (V2);

\node[right] (T) at (H2) {\footnotesize $J_1$};
\node[below left] (D) at (V2) {\footnotesize $J_2$};

\coordinate (PL) at (0.39*\maxx + 2.0, 0.3*\maxy + 2.0);
\draw[->, thick] (P) to (PL);
\node[right] (PLlb) at (PL) {\scriptsize $\bm{n}_{\lambda}$};

  \tkzDefPoint(0.38*\maxx, 0.31*\maxy){Ppgf}
  \tkzDefPoint(0.40*\maxx, 0.29*\maxy){Rpgf}
  \tkzDefLine[parallel=through Ppgf](Ppgf,Rpgf)
  \tkzDrawLine[add = 25 and 25, color=gray, dashed](Ppgf,tkzPointResult)
  
\end{tikzpicture}
&
\begin{tikzpicture}[
point/.style={draw,shape=circle,inner sep = 0mm,minimum size = .2cm},
scale = \tikzPfScale,
>=stealth',
shorten >=1pt,
shorten <=1pt,
auto,
node distance = 1.5cm,
semithick]

\def\maxx{5}
\def\maxy{5}

\coordinate (P) at (0.25*\maxx, 0.25*\maxy);

\node[point,fill=white] (Ppt) at (P) {};
\node[below left]       (Plb) at (P) {$S$};

\coordinate (X1) at (0.15*\maxx, 1.00*\maxy);
\coordinate (X2) at (0.25*\maxx, 0.25*\maxy);
\coordinate (Xn) at (1.00*\maxx, 0.15*\maxy);

\draw[thick]
	(X1) to[out=-89,in=100] 
	(X2) to[out=-10,in=179]
	(Xn);
	
\node[point,fill=black] at (X1) {};
\node[point,fill=black] at (Xn) {};

\coordinate (H1) at (-1,0);
\coordinate (H2) at (\maxx+1,0);
\coordinate (V1) at (0,-1);
\coordinate (V2) at (0,\maxy+1);

\draw[->] (H1) -- (H2);
\draw[->] (V1) -- (V2);

\node[right] (T) at (H2) {\footnotesize $J_1$};
\node[below left] (D) at (V2) {\footnotesize $J_2$};

  \tkzDefPoint(0.24*\maxx, 0.28*\maxy){Ppgf}
  \tkzDefPoint(0.26*\maxx, 0.20*\maxy){Rpgf}
  \tkzDefLine[parallel=through Ppgf](Ppgf,Rpgf)
  \tkzDrawLine[add = 10 and 5, color=gray, dashed](Ppgf,tkzPointResult)
  
  \tkzDefPoint(0.2*\maxx, 0.26*\maxy){Ppgf2}
  \tkzDefPoint(0.32*\maxx, 0.225*\maxy){Rpgf2}
  \tkzDefLine[parallel=through Ppgf2](Ppgf2,Rpgf2)
  \tkzDrawLine[add = 3 and 6, color=gray, dashed](Ppgf2,tkzPointResult)
  
  \tkzDefPoint(0.2*\maxx, 0.32*\maxy){Ppgf3}
  \tkzDefPoint(0.3*\maxx, 0.2*\maxy){Rpgf3}
  \tkzDefLine[parallel=through Ppgf3](Ppgf3,Rpgf3)
  \tkzDrawLine[add = 3.5 and 3, color=gray, dashed](Ppgf3,tkzPointResult)

\end{tikzpicture}
&
\begin{tikzpicture}[
point/.style={draw,shape=circle,inner sep = 0mm,minimum size = .2cm},
scale = \tikzPfScale,
>=stealth',
shorten >=1pt,
shorten <=1pt,
auto,
node distance = 1.5cm,
semithick,
 extended line/.style={shorten >=-#1,shorten <=-#1},
 extended line/.default=1cm,
 one end extended/.style={shorten >=-#1},
 one end extended/.default=1cm,]

\def\maxx{5}
\def\maxy{5}

\coordinate (P) at (0.125*\maxx, 0.635*\maxy);
\coordinate (R) at (0.7*\maxx, 0.16*\maxy);

\node[point,fill=white] (Ppt) at (P) {};
\node[below left]       (Plb) at (P) {$S$};
\node[point,fill=white] (Rpt) at (R) {};
\node[above right]      (Rlb) at (R) {$Q$};

\coordinate (X1) at (0.05*\maxx, 1.00*\maxy);
\coordinate (X2) at (0.20*\maxx, 0.60*\maxy);
\coordinate (X3) at (0.50*\maxx, 0.50*\maxy);
\coordinate (X4) at (0.60*\maxx, 0.35*\maxy);
\coordinate (Xn) at (1.00*\maxx, 0.05*\maxy);

\draw[thick]
	(X1) to[out=-80,in=170] 
	(X2) to[out=-10,in=135] 
	(X3) to[out=-45,in=100]
	(X4) to[out=-80,in=180]
	(Xn);
	
\node[point,fill=black] at (X1) {};
\node[point,fill=black] at (Xn) {};

\coordinate (H1) at (-1,0);
\coordinate (H2) at (\maxx+1,0);
\coordinate (V1) at (0,-1);
\coordinate (V2) at (0,\maxy+1);

\draw[->] (H1) -- (H2);
\draw[->] (V1) -- (V2);

\node[right] (T) at (H2) {\footnotesize $J_1$};
\node[below left] (D) at (V2) {\footnotesize $J_2$};

\draw[extended line, dashed, gray] (P) -- (R);

\end{tikzpicture}
\\
(A) & (B) & (C)
\end{tabular}
\end{center}
\caption{ 
Common Pareto Front scenarios.
	\textsc{(a)} Convex smooth PF
	with a point $S$ corresponding to some specific $\lambda$.
	The dashed line perpendicular to $\bm{n}_{\lambda} =  (\lambda, 1-\lambda)$ 
	is tangent to PF at $S$.
	If any path yielded a $(\J_1, \J_2)$ point below this line, the path corresponding to $S$ would not be $\lambda$-optimal. Since the grey rectangle (comprised of points dominating $S$) lies below that line,
	$\lambda$-optimality implies Pareto optimality $\forall \lambda \in (0,1).$
	\textsc{(b)} Convex non-smooth PF with a `kink' at the point $S$ makes the corresponding path $\lambda$-optimal for a range of $\lambda$'s , with a different ``support hyperplane'' corresponding to each of them.  
	\textsc{(c)} Non-convex smooth PF. Points $S$ and $Q$ correspond to 2 different $\lambda$-optimal paths.
	The portion of $PF$ between $S$ and $Q$ cannot be found by scalarization.
}
\label{fig:pf_scal}
\end{figure}

\begin{observ}
Any global maximizer of \eqref{eq:multi_objective} will correspond to a point on the convex part of PF.
\end{observ}

\begin{proof}
Suppose a control $\ba$ corresponds to a point $\left(\J_1(\ba),\J_2(\ba)\right)$ on a non-convex section of PF.
Then there exist controls $\ba_1, \ba_2$ and a constant $0 \leq \theta \leq 1$ such that
\begin{align*}
\theta \J_1(\ba_1) + (1-\theta)\J_1(\ba_2) &\le \J_1(\ba) \\
\theta \J_2(\ba_1) + (1-\theta)\J_2(\ba_2) &\le \J_2(\ba)
\end{align*}
with at least one of the inequalities strict.
Since $\F$ is decreasing and convex in both arguments,
\begin{align*}
\F(\J_1(\ba),\J_2(\ba)) &< \F\biggl(\theta \J_1(\ba_1) + (1-\theta)\J_1(\ba_2) \, ,\, \theta \J_2(\ba_1) + (1-\theta)\J_2(\ba_2)\biggr) \\
&\le \theta\F\biggl(\J_1(\ba_1),\J_2(\ba_1) \biggr)+ (1-\theta)\F\biggl(\J_1(\ba_2),\J_2(\ba_2)\biggr) \\
&\le \max_{i=1,2} \left\{\F\bigl(\J_1(\ba_i),\J_2(\ba_i)\bigr)\right\}
\end{align*}
Thus, such $\ba$ cannot be a maximizer in \eqref{eq:multi_objective}.
\end{proof}

\subsection{Scalarized Eikonal}
For any $\lambda \in [0,1]$, we define the scalarized running cost $K^{\lambda}$ as
\begin{equation}\label{eq:scalarized_cost}
K^{\lambda}(\x) = \lambda \psi(\x) + (1-\lambda)K(\x).
\end{equation}
A standard optimal control argument \cite{bardicapuzzodolcetta} shows that
the value function
\begin{equation}\label{eq:scalarized_value}
u^{\lambda}(\x) = \inf_{\adot \in \A} \biggl\{ \lambda \J_1(\x,\adot) + (1-\lambda) \J_2(\x,\adot)\biggr\}
\end{equation}
is the unique viscosity solution \cite{crandall1983viscosity} of the scalarized Eikonal equation:
\begin{align}\label{eq:scalarized_eikonal}
\begin{split}
f(\x) \left|\nabla u^{\lambda}(\x)\right| = K^{\lambda}(\x), \qquad &\x \in \domain \\
u^{\lambda}(\x) = 0, \qquad &\x \in \boundary.
\end{split}\end{align}

For every specific $\lambda$, we can recover the corresponding $\lambda$-optimal paths by gradient descent in $u^{\lambda}.$
We note that taking $\lambda = 0$ corresponds to detection-indifferent optimization, yielding the PDE \eqref{eq:Eikonal}, which we already needed to solve to find $R$ along with the optimal pre-extraction path from $\boundary$ to each site $\x_0$.  The case $\lambda = 1$ yields the PDE for recovering the time-indifferent least-detectable paths. 

\begin{remark}
\label{rem:P_lower_bound}
We note that a proper choice of $\lambda$ can also represent the PDE model based on the ``$\Psi \approx \J_1$'' approximation discussed in Remark \ref{rem:linearize}.
In particular, for $\lambda_{\sharp} = B(\x_0) / \left( B(\x_0) + 1 \right),$
the equation \eqref{eq:scalarized_value} can be rewritten as
$$\left( B(\x_0) + 1 \right) f(\x) \left|\nabla u^{\lambda_{\sharp}}(\x) \right| \, = \, B(\x_0) \psi(\x) + K(\x).$$
As a result, $\left( B(\x_0) + 1 \right) u^{\lambda_{\sharp}}(\x)$  is the viscosity solution of  \eqref{eq:linearizingEikonal}.  
When $\J_1 \ll 1,$ this yields 
``nearly optimal'' post-extraction paths.
Based on \eqref{eq:linearizedPaApprox}, we also know that 
\begin{equation}
\label{eq:linearized_P_estimate}
\PA(\x_0) \; \geq \;
P^{}_{\sharp}(\x_0) \, = \, 
B(\x_0) \, - \, \left( B(\x_0) + 1 \right) u^{\lambda_{\sharp}}(\x_0) \, - \, R(\x_0).
\end{equation}
\end{remark}

\subsection{Auxiliary PDEs}
However, we are ultimately interested in the expected payoff \eqref{eq:multi_objective}, not the value function $u^{\lambda}(\x)$.
Under suitable regularity assumptions on the speed $f$ and running cost $K$, it can be shown that an optimal control $\adot$ exists for every $\x \in \domain$, and the infimum in formula \eqref{eq:scalarized_value} can be replaced with a minimum \cite{bardicapuzzodolcetta}.
We define the set of $\lambda$-optimal controls $\A^{\lambda}_{\x}$ starting from $\x \in \domain$ as:
\begin{equation}\label{eq:lambda_optimal}
\A^{\lambda}_{\x} = \argmin_{\adot \in \A} \biggl\{ \lambda \J_1(\x,\adot) + (1-\lambda) \J_2(\x,\adot)\biggr\}
\end{equation}
We note that the optimal control $\adot$ will be unique at every point $\x$ where $u^{\lambda}(\x)$ is differentiable. 
Given our conditions on $f$ and $K$, $u^{\lambda}(\x)$ is Lipschitz-continuous, and so $\A^{\lambda}(\x)$ will be a singleton for almost every $\x \in \domain$.

We will now define the $\lambda$-optimal-restricted value functions $v^{\lambda}_1$ and $v^{\lambda}_2$ as:
\begin{align}
v^{\lambda}_1(\x)  &= \inf_{\adot \in \A^{\lambda}_{\x}} \bigl\{\J_1(\x,\adot)\bigr\} \\
v^{\lambda}_2(\x)  &= \inf_{\adot \in \A^{\lambda}_{\x}} \bigl\{\J_2(\x,\adot)\bigr\}
\end{align}
Given the solution $u^{\lambda}(\x)$ of \eqref{eq:scalarized_eikonal}, we can recover $v^{\lambda}_1$ and $v^{\lambda}_2$ by solving the following system of auxiliary linear PDEs 
introduced in \cite{mitchell2003continuous}:
\begin{align}\label{eq:auxiliary_system}\begin{split}
\nabla v^{\lambda}_1(\x) \cdot \nabla u^{\lambda}(\x) &= \frac{\psi(\x)K^{\lambda}(\x)}{f(\x)^2} \\
\nabla v^{\lambda}_2(\x) \cdot \nabla u^{\lambda}(\x) &= \frac{K(\x)K^{\lambda}(\x)}{f(\x)^2} 
\end{split}\end{align}
with $v^{\lambda}_1 = v^{\lambda}_2 = 0$ on $\boundary$.

Since we have shown in Section \ref{ss:PF} that the global optimizer of $P$ lies on the convex part of the $(\J_1,\J_2)$ Pareto front, it can be recovered through scalarization. 
More precisely, the optimal expected payoff satisfies
\begin{equation*}
\PA(\x) = \sup_{\lambda \in [0,1]} \left\{ B(\x)e^{-v^{\lambda}_1(\x)} - v^{\lambda}_2(\x)\right\}
\, - \, R(\x),
\end{equation*}
where the optimal $\lambda$ will be usually $\x$-dependent.

\section{Numerical Methods}
\label{s:Numerics}

Since the domain $\domain$ is meant to represent some protected area such as a national park, it will in general be irregularly shaped.
One way of handling this would be to discretize $\domain$ directly through the use of 
$\boundary$-conforming  triangulated meshes.
Efficient methods for numerically solving equation \eqref{eq:scalarized_eikonal} on such meshes are already known \cite{barth1998numerical, sethian2000fast, qian2007fast}.
For this paper, we will instead use a Cartesian grid $\X$ to discretize a rectangular region in $\R^2$ containing $\domain$.
Without loss of generality, suppose that $\domain$ is contained within a rectangle $\left[0, X_{max}\right] \times \left[0, Y_{max}\right]$.
Our grid spacing will be $\Delta x = X_{max}/ N_x$ in the $x$-direction, and $\Delta y = Y_{max}/ N_y$ in the $y$-direction.
The coordinates of our gridpoints will be $x_i = i\Delta x$ for $i=0,1,\dots N_x$ and $y_j = j\Delta y$ for $j=0,1, \dots N_y$. 
For Model A, we will also discretize $\lambda$ as $\lambda_k = k/N_{\lambda}$ for $k = 0,1,\dots N_{\lambda}$.
For Model G,  we will also discretize $\Btilde$ as 
$\Btilde_m = \Bmin + m (\Bmax - \Bmin)/N_{\Btilde},$ with $\Bmin = \min_{\x} B(\x),$
$\Bmax = \max_{\x} B(\x),$ and $m = 0,1,\dots N_{\Btilde}$.

\subsection{Scalarized Eikonal Equation (Model A)}
We will approximate our value function with
\begin{equation*}
U^k_{i,j} \approx u^{\lambda_k}(x_i,y_j)
\end{equation*}
and our $\lambda$-optimal-restricted value functions with
$$
V^{k,1}_{i,j} \; \approx \; v^{\lambda_k}_1(x_i,y_j), \qquad
V^{k,2}_{i,j} \; \approx \; v^{\lambda_k}_2(x_i,y_j).
$$

For each $\lambda_k$ value, we will need to solve the scalarized Eikonal equation \eqref{eq:scalarized_eikonal} and the coupled system of auxiliary linear PDEs \eqref{eq:auxiliary_system} simultaneously.
To approximate the gradients of the value functions, we will make heavy use of the following first-order difference operators:
\begin{align*}
D^{+x}_{i,j}[W] &= \frac{W_{i+1,j}-W_{i,j}}{\Delta x}, \qquad &D^{-x}_{i,j}[W] &= \frac{W_{i,j} - W_{i-1,j}}{\Delta x}, \\
D^{+y}_{i,j}[W] &= \frac{W_{i,j+1}-W_{i,j}}{\Delta y}, \qquad &D^{-y}_{i,j}[W] &= \frac{W_{i,j} - W_{i,j-1}}{\Delta y}.
\end{align*}
where $W$ could represent any of the value function approximations $U^k$, $V^{k,1}$, or $V^{k,2}$.

Since the scalarized Eikonal equation \eqref{eq:scalarized_eikonal} and the auxiliary PDEs \eqref{eq:auxiliary_system} are solved along the same characteristics, we will make sure that our approximations of the gradients are always using the same stencils.
The upwind difference operator $\D^x_{ij}$ in the $x$-direction will be:
\begin{align*}
\D^x_{i,j}[W] = 
\begin{cases}
D^{+x}_{i,j}[W], &\text{if} \quad  D^{+x}_{i,j}[U^k] \le \min\left( -D^{-x}_{i,j}[U^k], 0\right); \\
D^{-x}_{i,j}[W], &\text{if} \quad  -D^{-x}_{i,j}[U^k] < \min\left( D^{+x}_{i,j}[U^k], 0\right); \\
0, &\text{otherwise};
\end{cases}
\end{align*}
with $\D^y_{i,j}$ defined accordingly for the $y$-direction.
Note that the stencil being used will always be determined by $U^k$ regardless of whether $\D$ is being applied to $U^k$, $V^{k,1}$, or $V^{k,2}$.

Using these first-order difference operators, we can write down the first-order upwind discretization of our system of PDEs as:
\begin{align}\label{eq:discretized_system}\begin{split}
&f_{i,j} \left[ \left(\D^x_{i,j}[U^k]\right)^2 + \left(\D^y_{i,j}[U^k]\right)^2\right]^{1/2} =K^{\lambda_k}_{i,j} \\
& \D^x_{i,j}[U^k] \cdot \D^x_{i,j}[V^{k,1}] + \D^y_{i,j}[U^k] \cdot \D^y_{i,j}[V^{k,1}] = \frac{\psi_{i,j}K^{\lambda_k}_{i,j}}{f_{i,j}^2} \\ 
& \D^x_{i,j}[U^k] \cdot \D^x_{i,j}[V^{k,2}] + \D^y_{i,j}[U^k] \cdot \D^y_{i,j}[V^{k,2}] = \frac{K_{i,j}K^{\lambda_k}_{i,j}}{f_{i,j}^2}
\end{split}\end{align}
with the boundary conditions
\begin{equation*}
U^k_{i,j} = V^{k,1}_{i,j} = V^{k,2}_{i,j} = 0, \qquad \text{for} \; (x_i,y_j) \not\in\domain. \\
\end{equation*}

Since this is a monotone and consistent finite difference approximation of \eqref{eq:scalarized_eikonal}, $U^k$ will converge under grid refinement to the viscosity solution $u^{\lambda_k}$ \cite{barles1991convergence}.
But on any fixed grid, finding $U^k$ values is not straightforward since the above discretized system is coupled and nonlinear.  On a grid with $M$ gridpoints, straightforward Gauss-Jacobi iterations result in a high computational cost of $O(M^2),$ and Gauss-Seidel iterations with any fixed ordering of gridpoints are similarly inefficient. 

A significant speed up is gained by exploiting the causal properties of upwind difference operators. Based on the definition of  $\D^x$ and $\D^y,$ each $U^k_{i,j}$ depends only on the {\em smaller} neighboring $U^k$ values.  This {\em monotone causality} allows for the use of non-iterative Dijkstra-like ``Fast Marching'' methods
\cite{tsitsiklis1995efficient, sethian1996fast, sethian1999fast},
which compute $U^k$ in $O(M \log M)$ operations.  They mimic the logic of a classical shortest path algorithm on graphs \cite{dijkstra1959} and effectively de-couple this discretized system by determining the correct ordering of gridpoints/equations on the fly.     

Another popular approach for efficient Eikonal solvers is to use Gauss-Seidel iterations, but alternating through a list of direction-aligned grid orderings (from SW, SE, NE, and NW). 
Such ``Fast Sweeping'' techniques \cite{Boue98markovchain, zhao2005fast, tsai2003fast} aim to guess the direction of characteristics
for a part of the domain, recovering many correct $U^k$ values in each iteration.  These methods converge in $O(\omega M)$ operations, but the factor $\omega$ is not known a priori
and depends on the number of times the characteristics change direction from quadrant to quadrant.  Both classes of methods have their preferred sets of problems. E.g., Fast Sweeping is clearly better when characteristics are largely straight lines, while Fast Marching is typically preferable on domains with complex geometry or with rapidly changing PDE coefficients.
More recent hybrid (two-scale) methods aim to combine the best features of Marching and Sweeping.  We refer to \cite{ChacVlad1, ChacVlad2} for a detailed description, benchmarking, and comparison with other sequential and parallel Eikonal solvers.

While our own implementation is based on Fast Marching, other methods mentioned above can be similarly adapted to the current application.  
(Since the directions of derivative approximations for $V^{k,1}$  and $V^{k,2}$ are determined by $U^k$ values, 
these two grid functions are also easily computed simultaneously with $U^k$.)
Which method would be more efficient will depend on the geometry of $\domain$ and the properties of $\psi, K,$ and $f$.

\begin{remark}
\label{rem:lambdas}
The above grid functions are needed to predict the decisions of a rational extractor. 
For any single site $\x = (x_i, y_j),$  whether it will be extracted 
under Model A 
can be answered as soon as we find any $k$ such that  
$$
B_{i,j}e^{-V^{k,1}_{i,j}} \, - \, V^{k,2}_{i,j} \, - \, R_{i,j} \; > \; \tilde{p}.
$$
The exact $\PA(x_i,y_j)$ could be computed by any continuous optimization in $\lambda$.
But to approximate $\PA$ on the entire spatial grid $\X,$ it is more efficient to
compute $V^{k,1}$ and $V^{k,2}$ for all $k = 0,1,\dots N_{\lambda}$ and then use a 
$\lambda$-grid search:
\begin{equation}
\label{eq:lambda_discrete}
\PA_{i,j} \; \approx \; \max_{k} \left\{ B_{i,j} e^{-V^{k,1}_{i,j}} - V^{k,2}_{i,j}\right\}
\, - \, R_{i,j}.
\end{equation}
\end{remark}

\subsection{Randomly-terminated Eikonal equation (Model G)}

We will define $\Ubar$ as a grid function approximating the value function $\ubar$ of section \ref{ss:post-extraction_G}.
The upwind difference operators can be similarly used to  approximate the gradient
and then used to obtain a consistent, monotone discretization of PDE \eqref{eq:ubar} :
\begin{equation}
\label{eq:discretized_system_G}
\begin{split}
\left[  \min\left( D^ {+x}_{i,j}[\Ubar^m], \, -D^{-x}_{i,j}[\Ubar^m], \, 0 \right)  \right]^2 \, &+ \, \left[  \min\left( D^ {+y}_{i,j}[\Ubar^m], \, -D^{-y}_{i,j}[\Ubar^m], \, 0 \right)    \right]^2\\ 
\, & = \, 
\left[ K_{i,j} \, + \, \psi_{ij}  \left( b_m + R_{i,j} - \Ubar^m_{i,j} \right) \right]^2 \, / \, f_{i,j}^2,
\end{split}
\end{equation}
with the boundary conditions $\Ubar^m_{i,j} = 0$ for $\, (x_i,y_j) \not\in\domain.$

This is essentially the same discretization previously used in \cite{AndrewsVlad}, where it was shown that its causal properties allow for a non-iterative Dijkstra-like  algorithm with $O(M \log M)$ complexity. We refer readers to  \cite[Appendix~B]{AndrewsVlad} for a quadrant-by-quadrant interpretations of \eqref{eq:discretized_system_G} and a detailed discussion of efficient implementation.  
This approach is arguably even better suited in the current context:
the ``free boundary'' present in \cite{AndrewsVlad} made it necessary to use a special initialization, but our Dirichlet boundary conditions make it possible to use the original Fast Marching. 

\begin{remark}
\label{rem:Btilde_values}
For any single site $\x = (x_i, y_j),$
the expected profit under Model G could be computed by using the exact $B_{i,j}$ instead of $b_m$
in equation \eqref{eq:discretized_system_G}.
But to approximate $\PG$ on the entire spatial grid $\X,$ it is more efficient to
compute $\Ubar^m$ for all $m = 0,1,\dots N_{b}.$ 
If $m$ is such that $B_{i,j} \in (b_m, b_{m+1}],$
we can define 
$$ P^{g-}_{i,j} = B_{i,j}  - \Ubar_{i,j}^m - R_{i,j} \quad \text{ and } \quad
P^{g+}_{i,j} = B_{i,j}  - \Ubar_{i,j}^{m+1} - R_{i,j}.
$$
Either one of these can be used as a good approximation since $\PG_{i,j} \in \left[ P^{g-}_{i,j}, \, P^{g+}_{i,j} \right]$ and 
$P^{g+}_{i,j} - P^{g-}_{i,j} \leq (\Bmax-\Bmin) / N_{b}.$  
We know that $(x_i, y_j)$ will be extracted if 
$P^{g-}_{i,j} > \tilde{p},$  and will not to be extracted if $P^{g+}_{i,j} \le \tilde{p}.$
In our numerical tests we use the $\tilde{p}$ level set of $\frac{1}{2} \left(P^{g-}+P^{g+} \right)$
to approximate $\pBoundaryG.$  A more conservative discrete approximation of that boundary could be also obtained
from the set
$\left\{ (x_i, y_j) \, \mid  \,   P^{g-}_{i,j}  \le \tilde{p} \le P^{g+}_{i,j} \right\}.$
\end{remark}

\section{Examples}
\label{s:Examples}

In all of our numerical examples, we use a homogeneous running cost $K(\x) = \kappa = 1.$
As a result, the pre-extraction cost $R(\x)$ coincides with $\tau(\x)$, the minimum time required starting from $\boundary$ to reach the extraction point $\x$.  
Similarly, $\J_2(\x,\adot)$ represents the 
time to  $\boundary$ starting from $\x$ and using the post-extraction control $\adot$. 
In Examples 1-5 the perpetrators can move with speed $f=1$ (and thus, $\tau(\x) = d(\x)$), while in Examples 6 and 7, their speed is terrain dependent. 

To save space,  
we primarily focus on the expected profit $\PA$ corresponding to Model A (the aerial patrols). Our Figure captions report the maximum attainable expected profit 
$\PbarA = \max_{\x \in \domain} \PA(\x),$ the proportion of protected area $\ApA = | \pDomainA | / |\domain|,$ 
and the proportion of protected value $\VpA = \int_{\pDomainA} B(\x) d\x \, / \, \int_{\domain} B(\x) d\x.$ 
Implementing the same patrol density $\psi$ in both 
Model A and Model G (the ground patrols) might be hard in practice, but for the sake of quantitative comparison, we assume that this can be done. 
Since in Model G the costs resulting from a detection are typically lower, this results in $\PG \geq \PA$ and $\pDomainG \subset \pDomainA.$
We demonstrate this in Examples 5-7,
where we compare $\pBoundaryA$ with  $\pBoundaryG$  (using $N_{\lambda} = N_{\Btilde}$)  and provide the related statistics.

In most cases we consider $\psi$ to be a part of the problem statement, but in Examples 3 and 4 we also perform a straightforward optimization of $\psi$ to maximize $\ApA$ while respecting the specified patrol density budget $E$, using \eqref{eq:patrol_budget} with $\gamma = 1$.

Our numerical approach is quite efficient.  Even though the implementation is not highly optimized, it was easy to perform all simulations on a mid-range laptop computer (MacBook Pro with a 2GHz Intel Core i5 processor and 8GB of RAM).
For a fixed $\psi$ and any specific $\lambda,$ the execution times varied from 120 milliseconds on smaller grids (Examples 3-4) to 50 seconds on larger grids (Examples 6-7).  The overall time depended on the desired $\lambda$-resolution. The source code for our implementation will be made available at \href{https://github.com/eikonal-equation/Environmental-Crime}{\url{https://github.com/eikonal-equation/Environmental-Crime}}.

\subsection{Example 1: a circle}
We start with a radially symmetric example close to the original setting of Albers  \cite{albers2010spatial}.
The domain $\domain$ is a disk of diameter 1, and we use a constant benefit function $B=2.$
But we assume a more general pointwise detection rate based on the distance $d$ from the boundary:
\begin{equation}
\label{eq:banded_psi}
\psi(d) = \frac{\mu}{50(d-0.3)^2 + 0.5} ,  \qquad d \in [0,\, 0.5],
\end{equation}
where $\mu$ is a constant chosen to enforce the budget constraint $\int_{\domain} \psi(\x) d\x = E$ with $E=2.5.$
A contour plot of $\psi$ can be found in Figure \ref{fig:ex1_psi}.
The expected profit $\PA$ is shown in Figure \ref{fig:ex1_P}, with the dashed line indicating the boundary of the pristine region $\Omega_p$. 
The dotted line indicates the zero level set of $P_{\sharp}$, the approximation to $\PA$ discussed in Remark \ref{rem:P_lower_bound}.

We use a uniform rectangular grid with dimensions $N_x=501$, $N_y = 501$.
Since $f$ is uniformly one, with $\psi$ and $K$ depending only on the distance from the boundary, this example falls into the category of ``$\tau$-banded'' scenarios discussed in Section \ref{s:MultiObjective}.
The optimal trajectories are simply straight lines to the closest point on the boundary.
Since the same trajectory minimizes both $\J_1$ and $\J_2$, the Pareto Front consists of a single point, which can be recovered using any $\lambda$ with a single PDE solve.
\begin{figure}
\begin{center}
\begin{subfigure}[t]{0.32\textwidth}
	\includegraphics[height=0.8\textwidth]{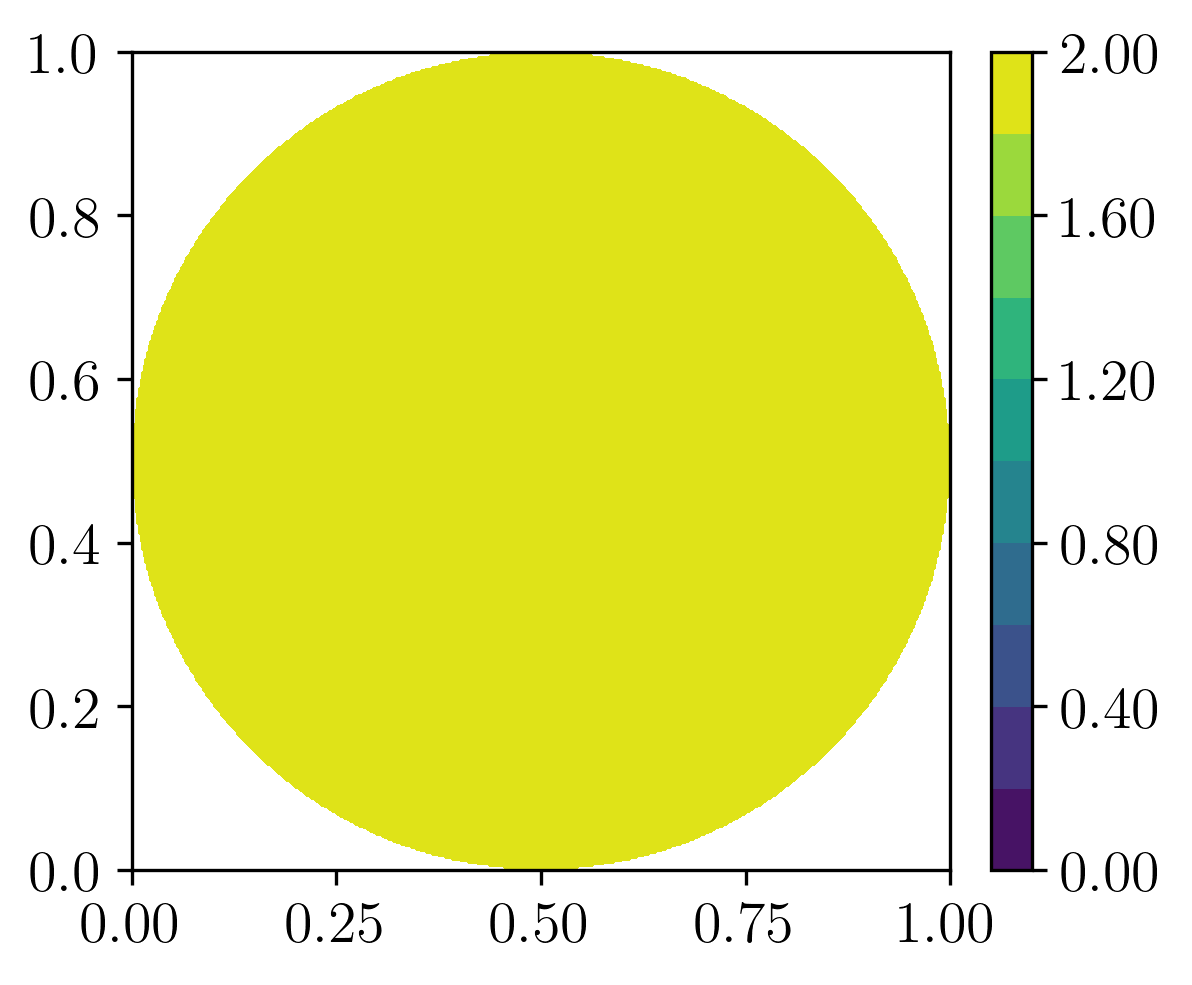}
	\subcaption{Homogeneous benefit function $B(x,y)$}\label{fig:ex1_B}
\end{subfigure}
\begin{subfigure}[t]{0.32\textwidth}
	\includegraphics[height=0.8\textwidth]{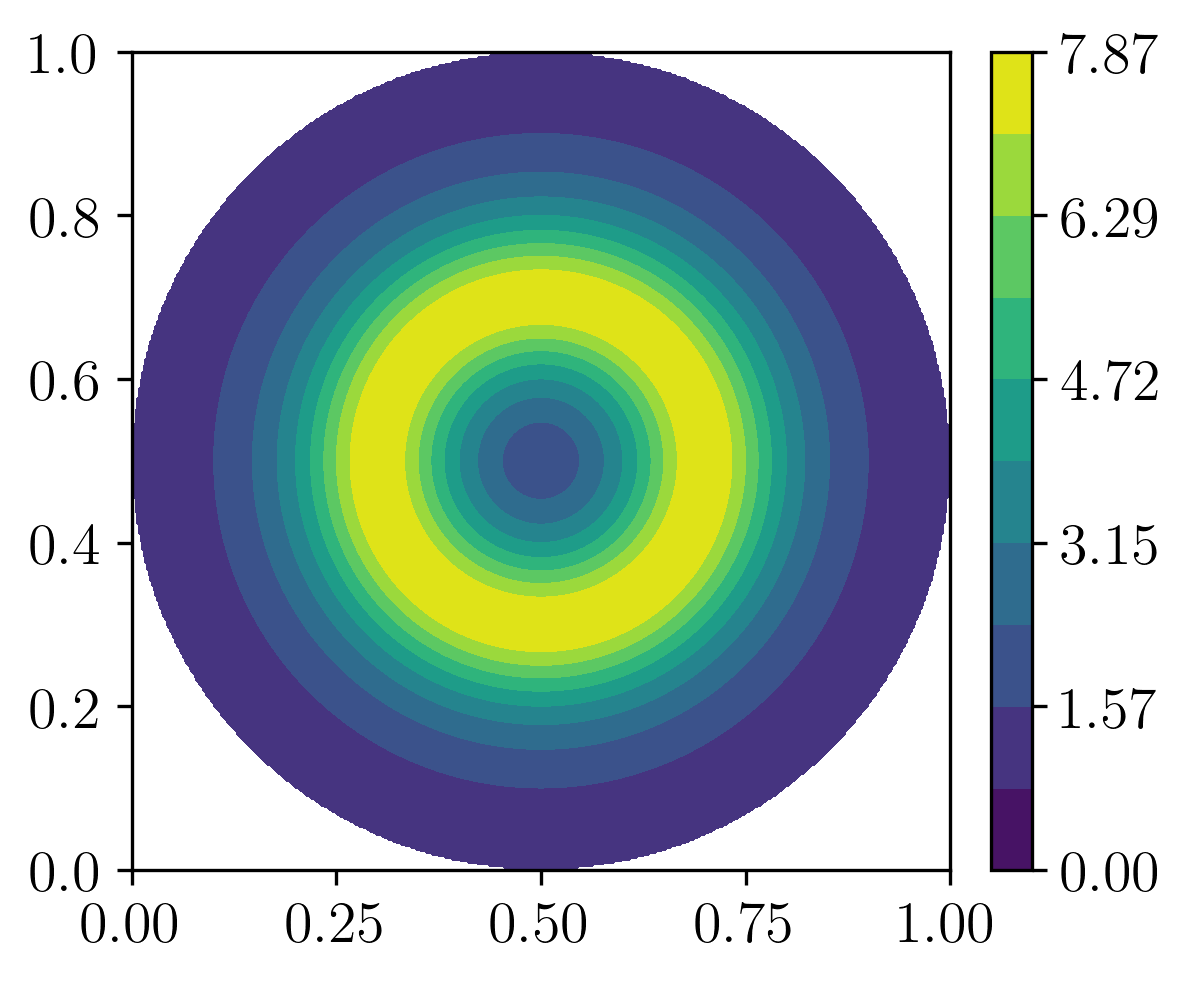}
	\subcaption{\centering Radially symmetric patrol density function $\psi(x,y)$}\label{fig:ex1_psi}
\end{subfigure} 
\begin{subfigure}[t]{0.32\textwidth}
	\includegraphics[height=0.8\textwidth]{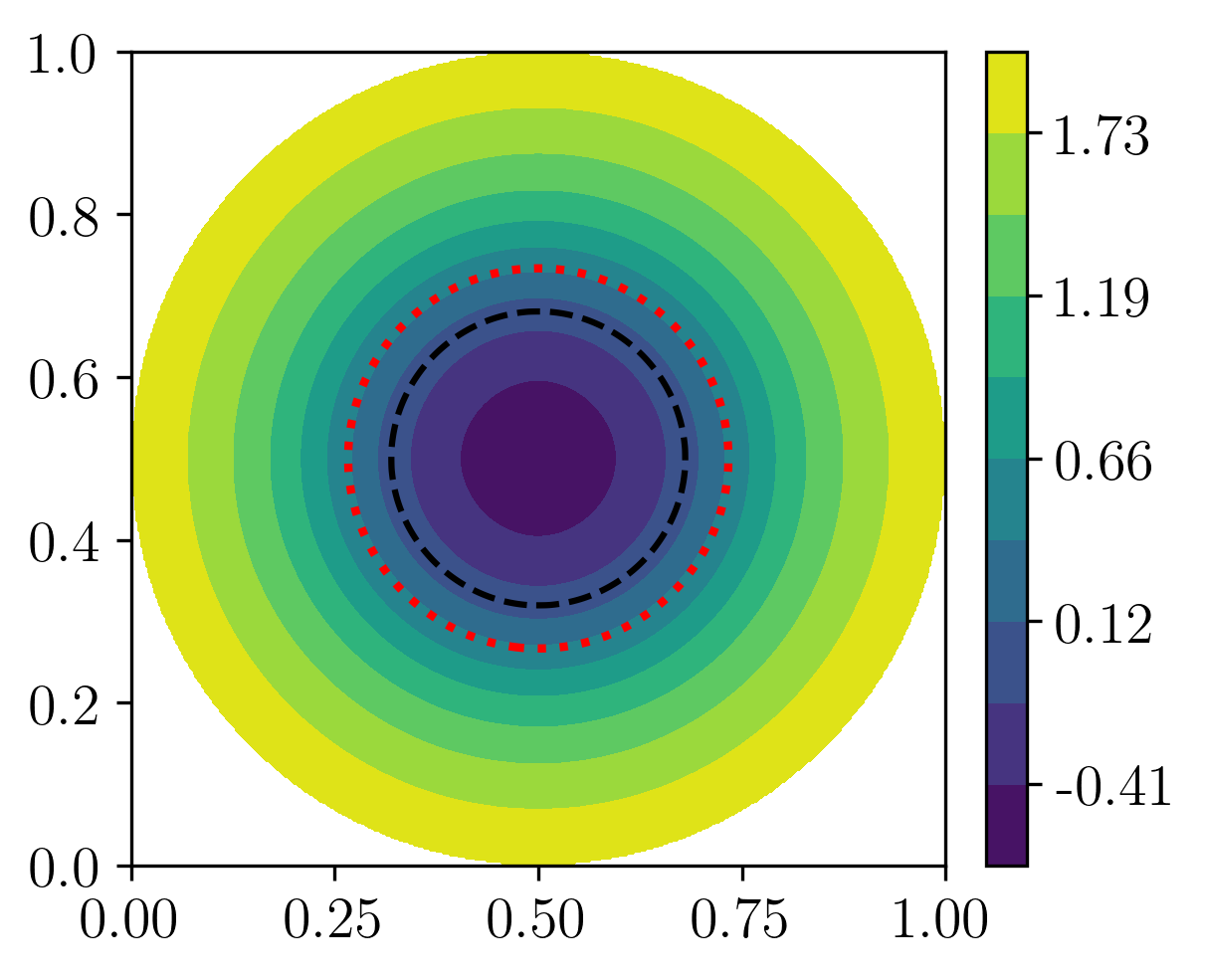}
	\subcaption{Expected profit $\PA(x,y)$ for the extractor}\label{fig:ex1_P}
\end{subfigure}
\caption{A disk-like domain. The maximum profit $\PbarA=2$ is attained at $\boundary$, the pristine area proportion is $\ApA = 13.02\%$ and, since $B(\x)$ is constant,  the same $\VpA = 13.02\%$ of the value is protected. If $P_{\sharp}$ was used instead of $\PA$, it would overestimate the pristine proportion and value protected as $\ApA = \VpA = 21.77\%$.}\label{fig:ex1}
\end{center}
\end{figure}

\subsection{Example 2: banded patrol density on a square}
We use a sum of two Gaussians to specify the benefit function
\begin{equation*}
B(x,y) = e^{-10\left((x-0.25)^2+(y-0.5)^2\right)} + e^{-10\left((x-0.75)^2+(y-0.5)^2\right)}
\end{equation*}
on a unit square $\domain;$ see Figure \ref{fig:ex2_B}.
We assume the same ``banded'' detection rate defined in \eqref{eq:banded_psi}, but
with $\mu$ selected based on $E=2.$
A contour plot of $\psi$ is shown in Figure \ref{fig:ex2_psi}.
The expected profit $\PA$ can be found in Figure \ref{fig:ex2_P}, with the dashed line showing 
$\pBoundaryA$.

This example was computed on a uniform rectangular grid with dimensions $N_x=501$, $N_y = 501$.
As in Example 1, this is another case of a ``$\tau$-banded'' patrol density, so optimal trajectories are straight lines to the closest point on $\boundary$, and any single $u^\lambda$ is sufficient to compute $\PA(\x)$.
(Note that this property holds regardless of a particular benefit function $B.$ It also implies that both Models A and G will yield exactly the same predictions.)
\begin{figure}
\centering
\begin{subfigure}[t]{0.32\textwidth}
	\includegraphics[height=0.8\textwidth]{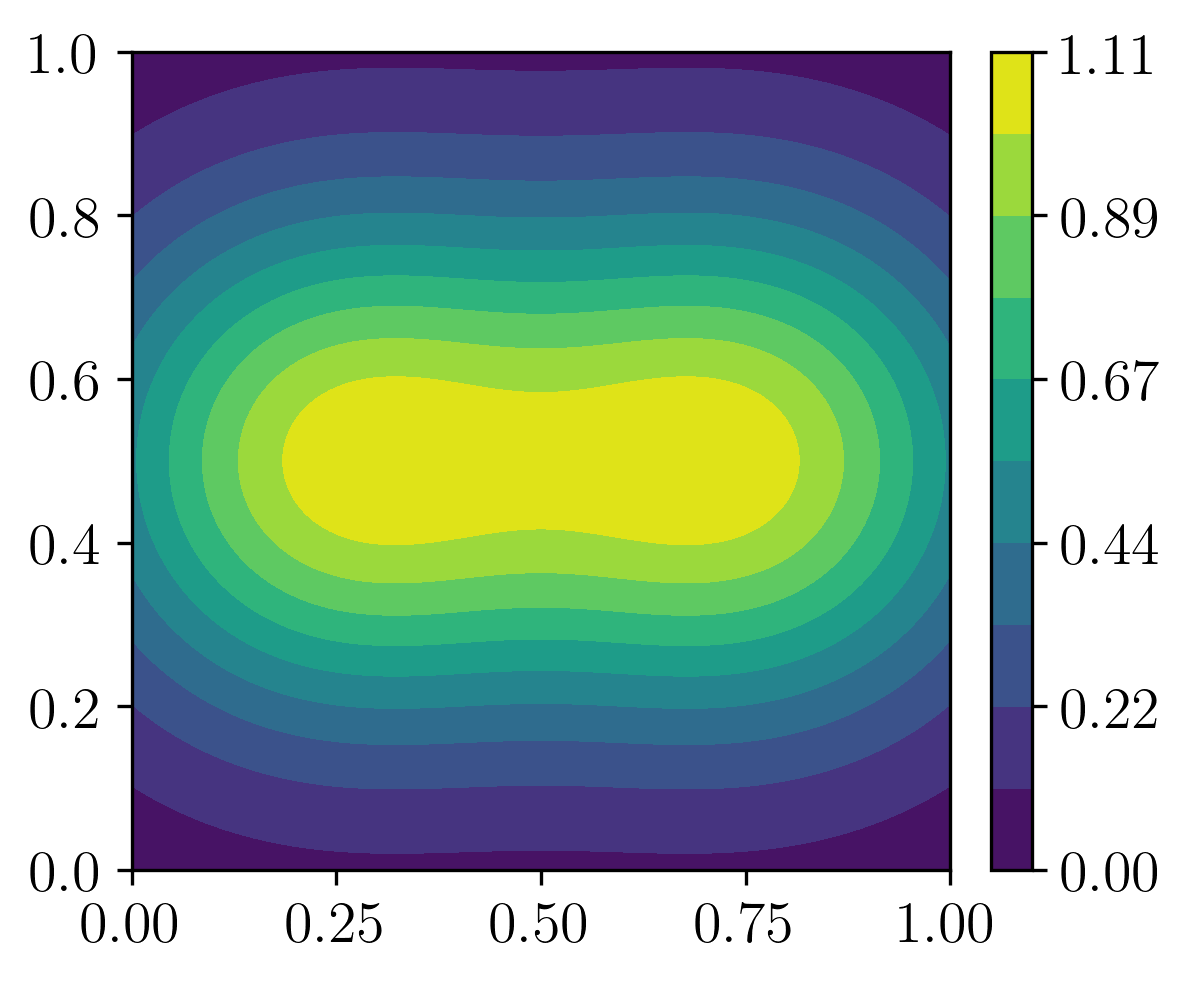}
	\subcaption{\centering Benefit function $B(x,y)$}\label{fig:ex2_B}
\end{subfigure}
\begin{subfigure}[t]{0.32\textwidth}
	\includegraphics[height=0.8\textwidth]{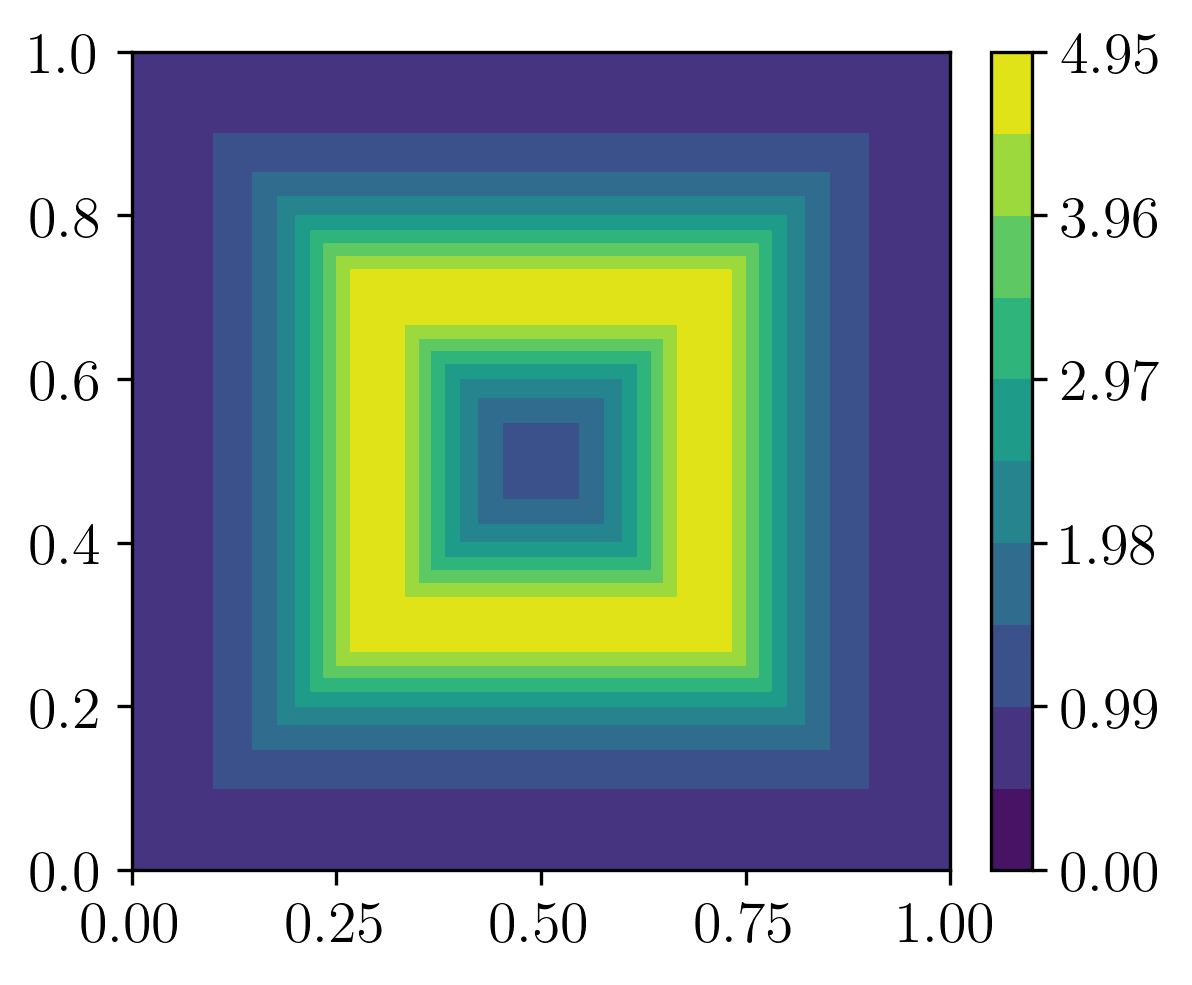}
	\subcaption{\centering Banded 
	$\psi(x,y)$}\label{fig:ex2_psi}
\end{subfigure} 
\begin{subfigure}[t]{0.32\textwidth}
	\includegraphics[height=0.8\textwidth]{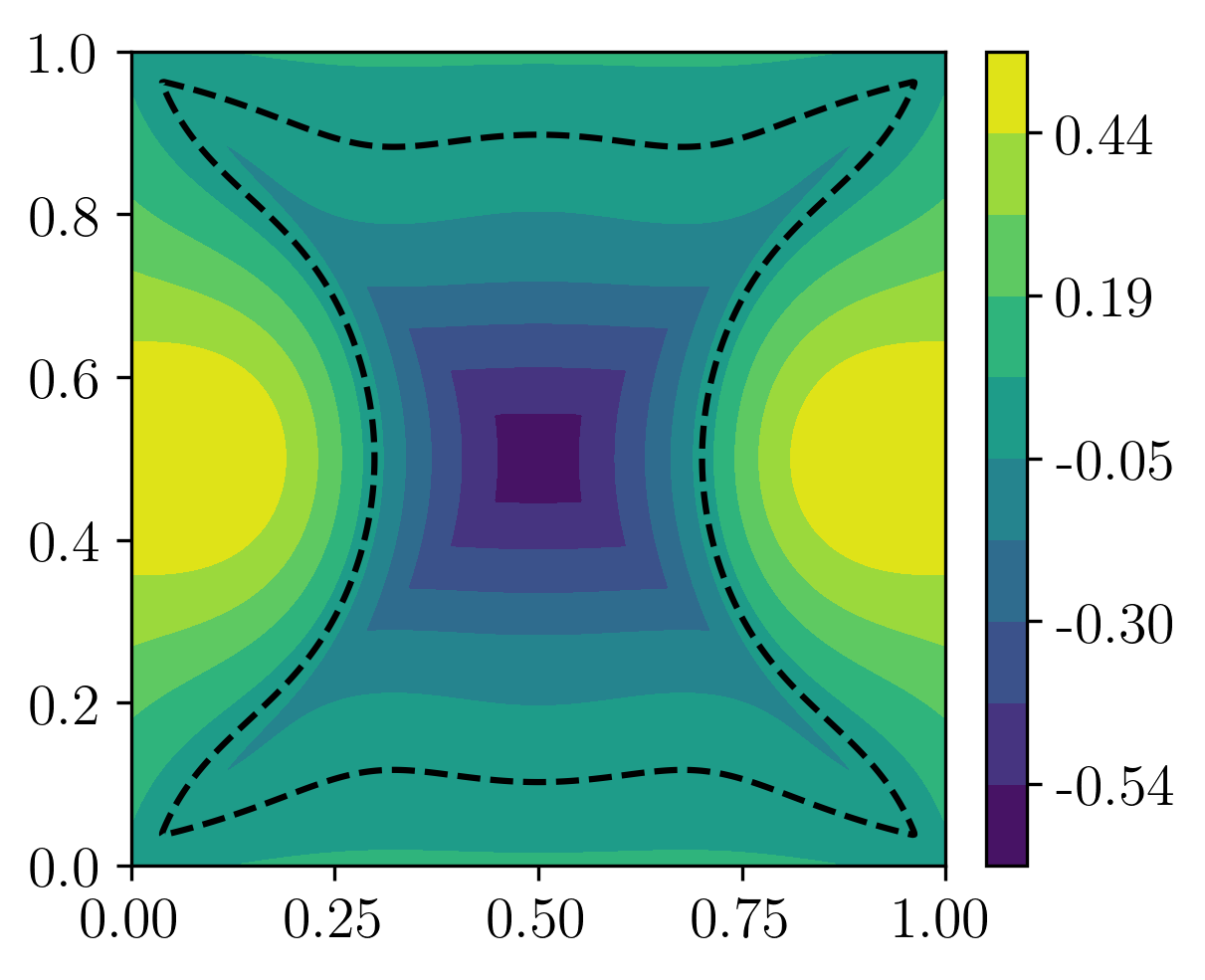}
	\subcaption{\centering Expected profit $\PA(x,y)$}\label{fig:ex2_P}
\end{subfigure}
\caption{Banded patrol density on a square. $\PbarA = 0.56, \, \ApA = 45.67\%,$ and $\VpA=52.99\%.$}\label{fig:ex2}
\end{figure}

\subsection{Example 3: placing a patrol station}
As noted in \cite{albers2010spatial}, in realistic situations PAMs face additional constraints 
and the detection rate typically decreases with the distance to  patrol stations (in our case -- the locations where the surveillance drones are based).
We will suppose that the detection rate $\psi(x,y)$ is now a single Gaussian ``centered'' at 
a patrol station $(\bar{x}, \bar{y}) \in \domain$:
\begin{equation*}
\psi_{\bar{x},\bar{y}}^{} (x,y) = \mu e^{-30\left((x-\bar{x})^2 + (y-\bar{y})^2\right)}
\end{equation*}
with $\mu$ selected based on $E = 2$. 
Using the same domain, speed, and benefit functions defined in Example 2, we will now search for the patrol station location that maximizes the pristine area proportion $\ApA.$
In principle, this could be done with any black-box derivative-free optimization method on $\domain,$ but for the sake of simplicity we have implemented a grid search on a coarse 
($11\times11$) grid.  For each candidate location, the pristine region $\Omega_p$ is computed on a uniform rectangular grid with $N_x=201$, $N_y = 201$, and $N_{\lambda}=101$.

Due to the symmetry, there are actually two $\ApA$-maximizing locations: $(\bar{x}, \bar{y}) = (0.5, 0.3)$ and $(\bar{x}, \bar{y}) = (0.5, 0.7).$  Figure \ref{fig:ex3_psi} shows $\psi$ centered at the former.
The corresponding expected profit $\PA$ can be found in Figure \ref{fig:ex3_P}, with the dashed line showing $\pBoundaryA$.

\begin{figure}
\begin{subfigure}[t]{0.5\textwidth}
	\includegraphics[height=0.7\textwidth]{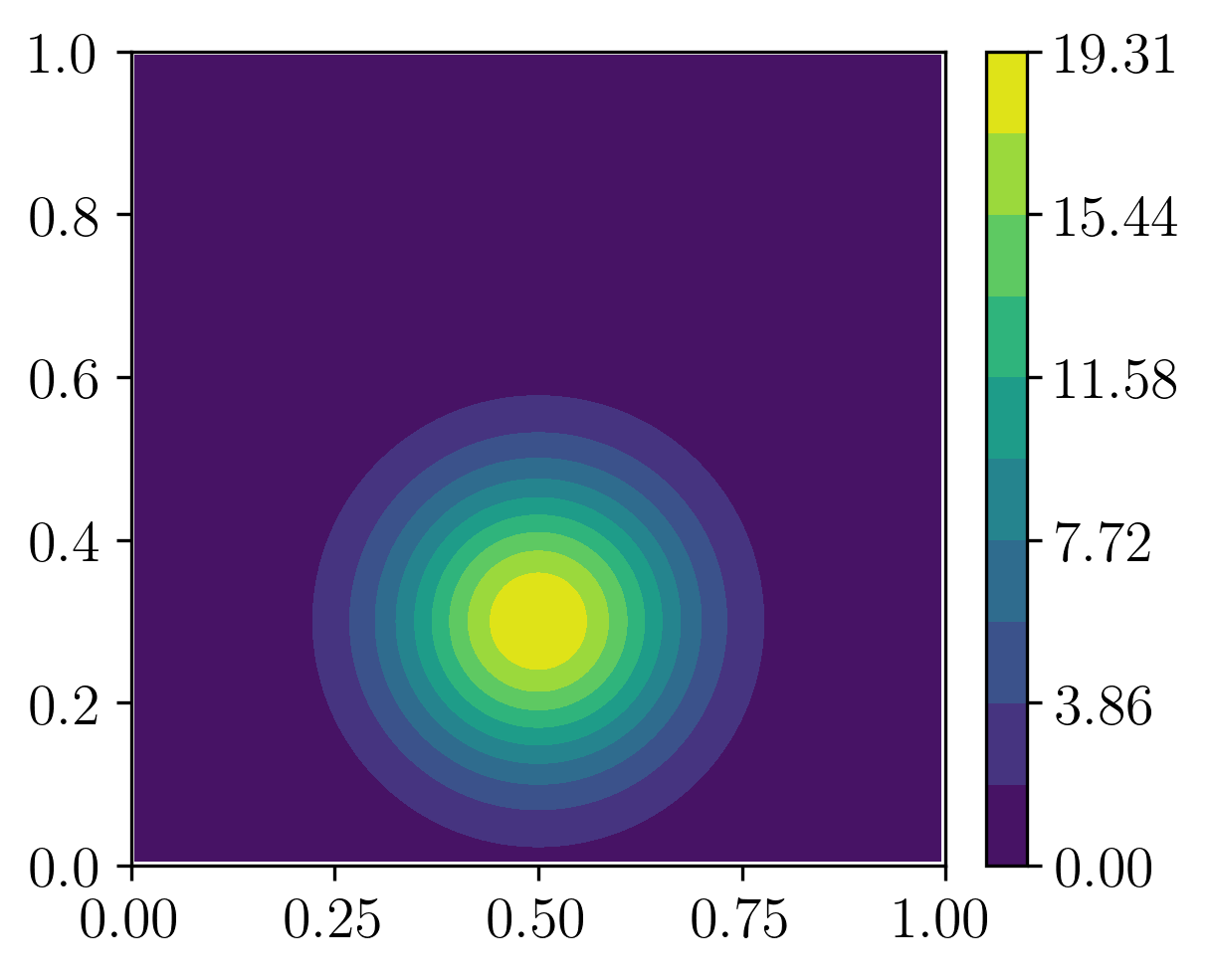}
	\subcaption{\centering $\psi(x,y)$ centered at $(\bar{x}, \bar{y}) = (0.5, 0.3)$}\label{fig:ex3_psi}
\end{subfigure} 
\hspace{-0.15in}
\begin{subfigure}[t]{0.5\textwidth}
	\includegraphics[height=0.7\textwidth]{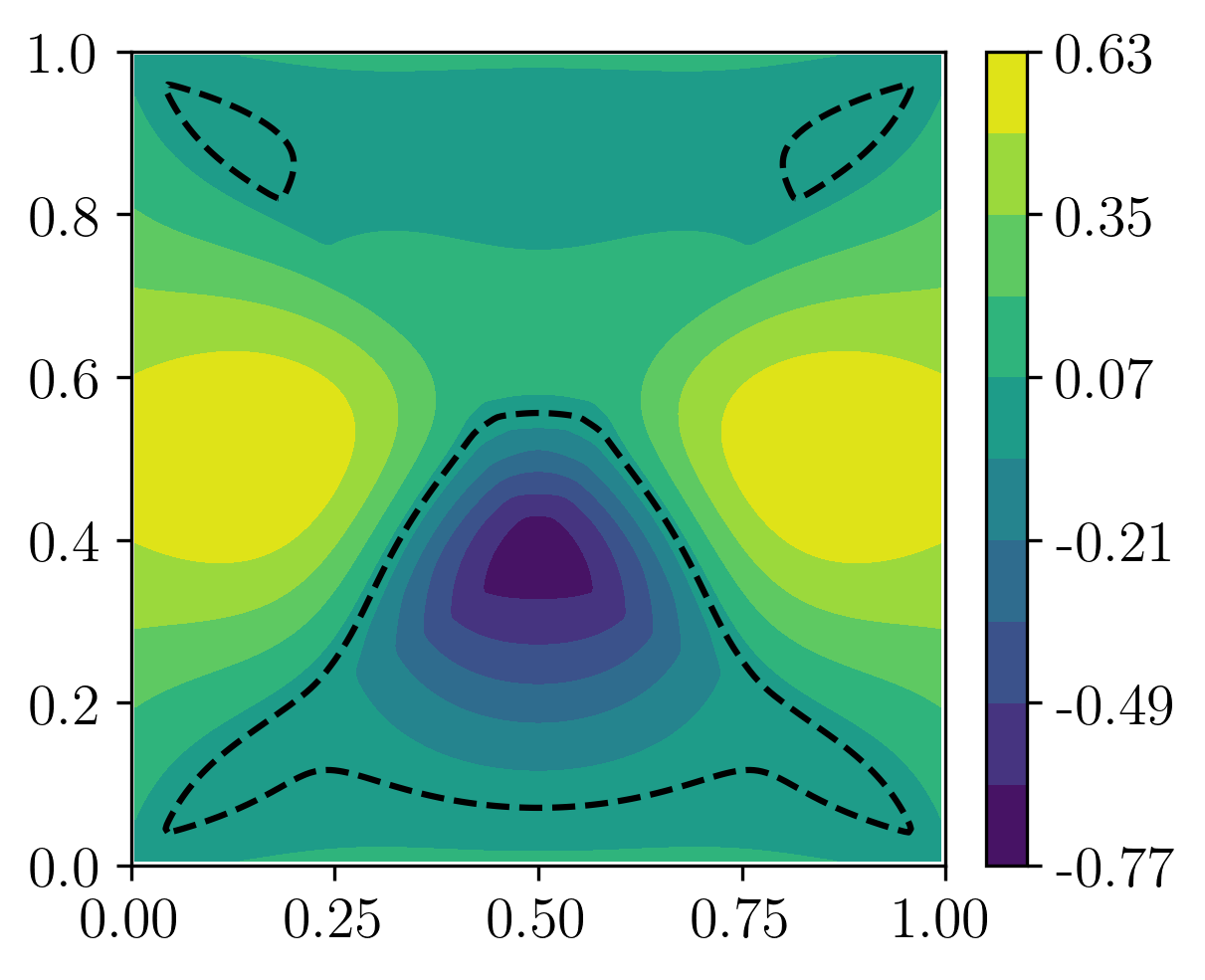}
	\subcaption{\centering Expected profit $\PA(x,y)$}\label{fig:ex3_P}
\end{subfigure}
\caption{Optimal patrol station position. $\PbarA =  0.63, \, \ApA =  23.53\%,$ and $\VpA = 24.54\%$}\label{fig:ex3}
\end{figure}

\subsection{Example 4: allocating surveillance resources}
Given the non-unique maxima $(\bar{x}, \bar{y})$ in Example 3, we now suppose that PAMs decide to build a separate drone station at each of these two locations.   The next question is how to share the surveillance resources 
between them.  Assuming that the overall number of drones is fixed and that the detection rate generated by each station is proportional to the number of drones stationed there, any allocation corresponds to choosing $\textbf{w}=(w_1,w_2)$ with $w_1,w,2 \geq 0, \; w_1+w_2 = 1,$ and defining
\begin{align*}
\psi_{\textbf{w}}(x,y) &= \mu_{\textbf{w}} \left( w_1 G_1(x,y) + w_2 G_2(x,y) \right),\\
G_1(x,y) &= e^{-30((x-0.5)^2+(y-0.3)^2)}, \\
G_2(x,y) &= e^{-30((x-0.5)^2+(y-0.7)^2)},
\end{align*}
with $\mu_{\textbf{w}}$ chosen to ensure $\int_{\domain} \psi_{\textbf{w}}(x,y) \, dx \, dy = E = 2 $. 
For each $\textbf{w},$ we compute the $\ApA$ on a uniform rectangular grid with  $N_x=201$, $N_y = 201$, and $N_{\lambda}=101$. 
A grid search over 101 values of $w_1$ is used to find the $\ApA$-maximizing allocation. 
As in the previous example, the answer is not unique: both $\textbf{w}=(0.43, 0.57)$ and $\textbf{w}=(0.57, 0.43)$ are optimal.

For the first of these, Figure \ref{fig:ex4_psi} shows a contour plot of $\psi_{\textbf{w}}$ along with four different Pareto-optimal paths starting from an extraction point $\x_0 = (0.555, 0.315)$.
The path maximizing the expected profit $\PA$ is shown in magenta, the path minimizing the probability of detection $\Psi$ is shown in green, and the path minimizing the time of travel is shown in black.
The red path corresponds to the ``nearly optimal''  $\lambda_{\sharp}= \frac{B}{B+1}$ of the linearized-$\Psi$ problem, as discussed in Remark \ref{rem:P_lower_bound}.
The corresponding expected profit $\PA$ can be found in Figure \ref{fig:ex4_P}, with the dashed line showing $\pBoundaryA$.
\begin{figure}
\begin{center}
\begin{subfigure}[t]{0.5\textwidth}
	\includegraphics[height=0.7\textwidth]{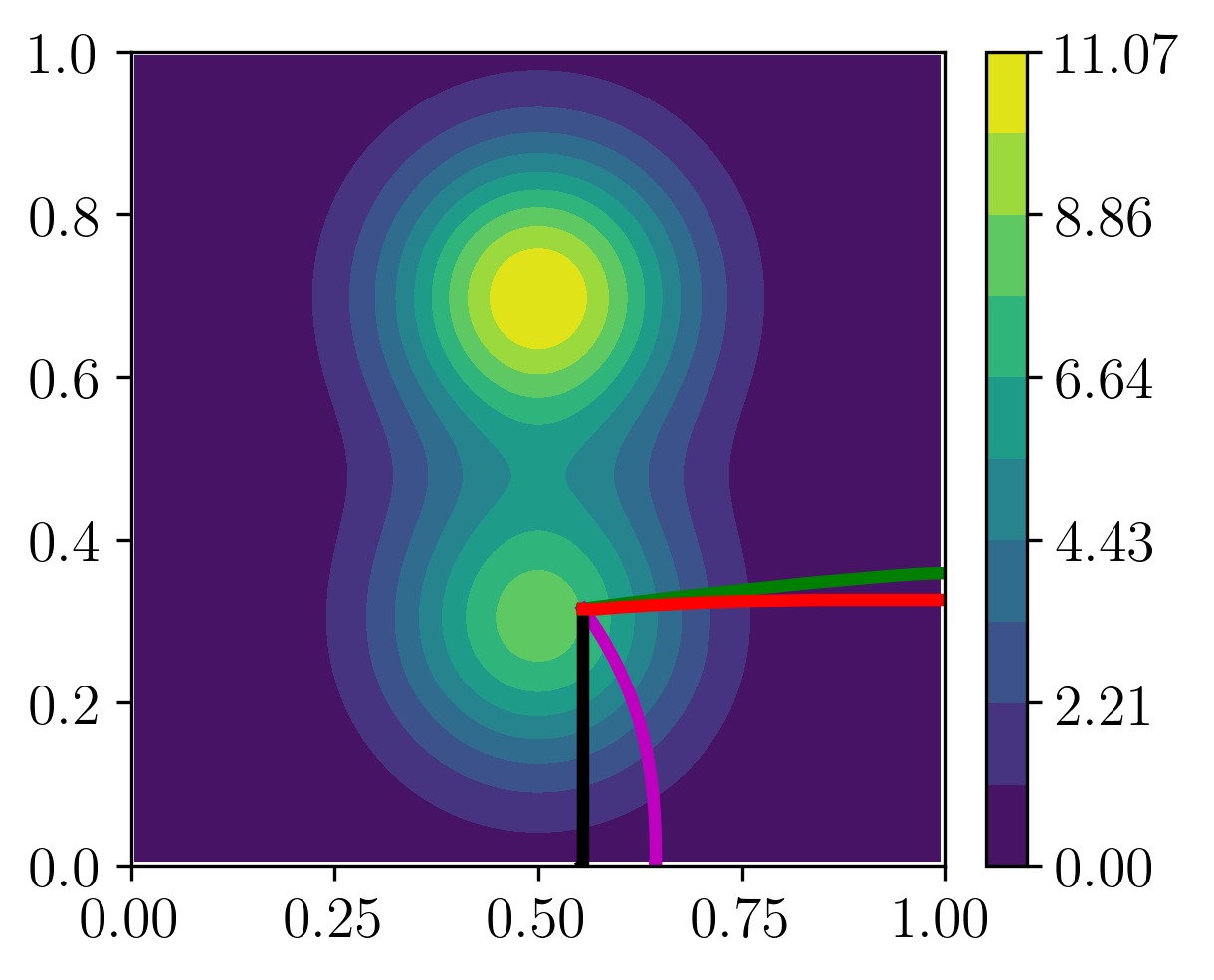}
	\subcaption{\centering Patrol density function $\psi(x,y)$ with the optimal weights $\textbf{w}=(0.43, 0.57)$ }\label{fig:ex4_psi}
\end{subfigure} 
\hspace{-0.15in}
\begin{subfigure}[t]{0.5\textwidth}
	\includegraphics[height=0.7\textwidth]{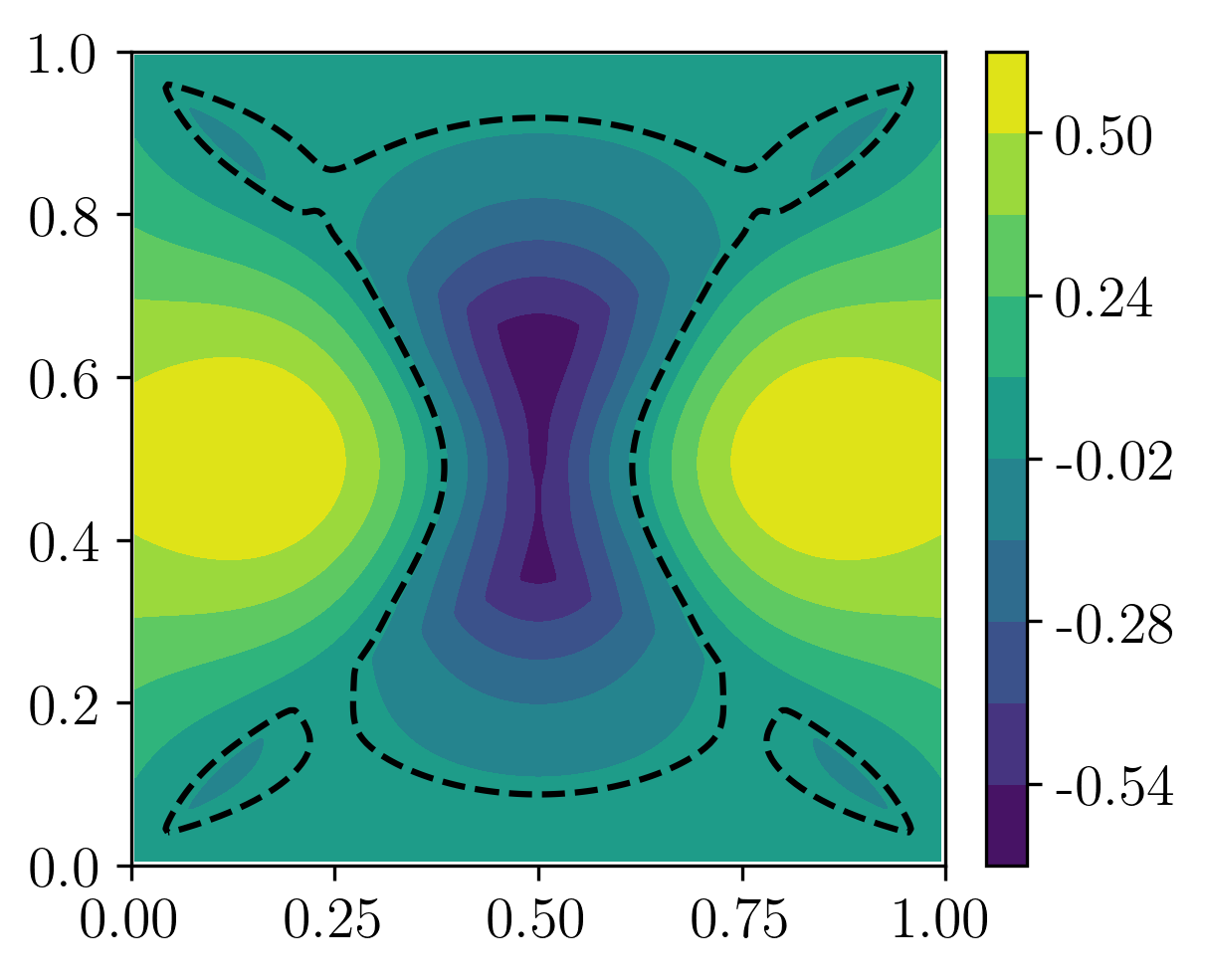}
	\subcaption{\centering Expected profit $\PA(x,y)$ for the extractor}\label{fig:ex4_P}
\end{subfigure}
\caption{Allocating surveillance resources. $\PbarA =  0.63$, $\ApA$ has increased to 35.26\%, and $\VpA$ has increased to 38.48\%. We note that from the PAMs' perspective, the banded patrol from Example 2 still performs better while using the same budget ($E=2$).}\label{fig:ex4}
\end{center}
\end{figure}

In Figure \ref{fig:lambda}, we consider the $\lambda$-optimization problem \eqref{eq:lambda_discrete} for the same extraction point $\x_0$ in detail.
Figure \ref{fig:lambda_optimization} shows the $\lambda_k$-optimal payoffs $B(x,y)e^{-V^{k,1}_{i,j}} - V^{k,2}_{i,j} - R(x,y)$ in blue, with the maximum profit $\PA(x,y)$) as a black dashed line.
The choice of path based on $\lambda_{\sharp}$ is clearly far from optimal, which is not surprising since $\mathcal{J}_1$ is not small.  
 The linearized-$\Psi$ approach significantly overestimates the probability of detection along all trajectories.  In fact, if we were to base the estimated ``profit'' on $u^{\lambda_{\sharp}},$
the formula \eqref{eq:linearized_P_estimate} would yield an even more pessimistic $P^{}_{\sharp}(\x_0) = - 0.699.$

Figure \ref{fig:PF} shows the $(\J_1,\J_2)$ Pareto Front for the same $\x_0.$  
The color of points in in \ref{fig:lambda_optimization} and \ref{fig:PF} matches  the color of the corresponding paths in \ref{fig:ex4_psi}.
\begin{figure}
\begin{center}
\begin{subfigure}[t]{0.48\textwidth}
	\includegraphics[height=0.7\textwidth]{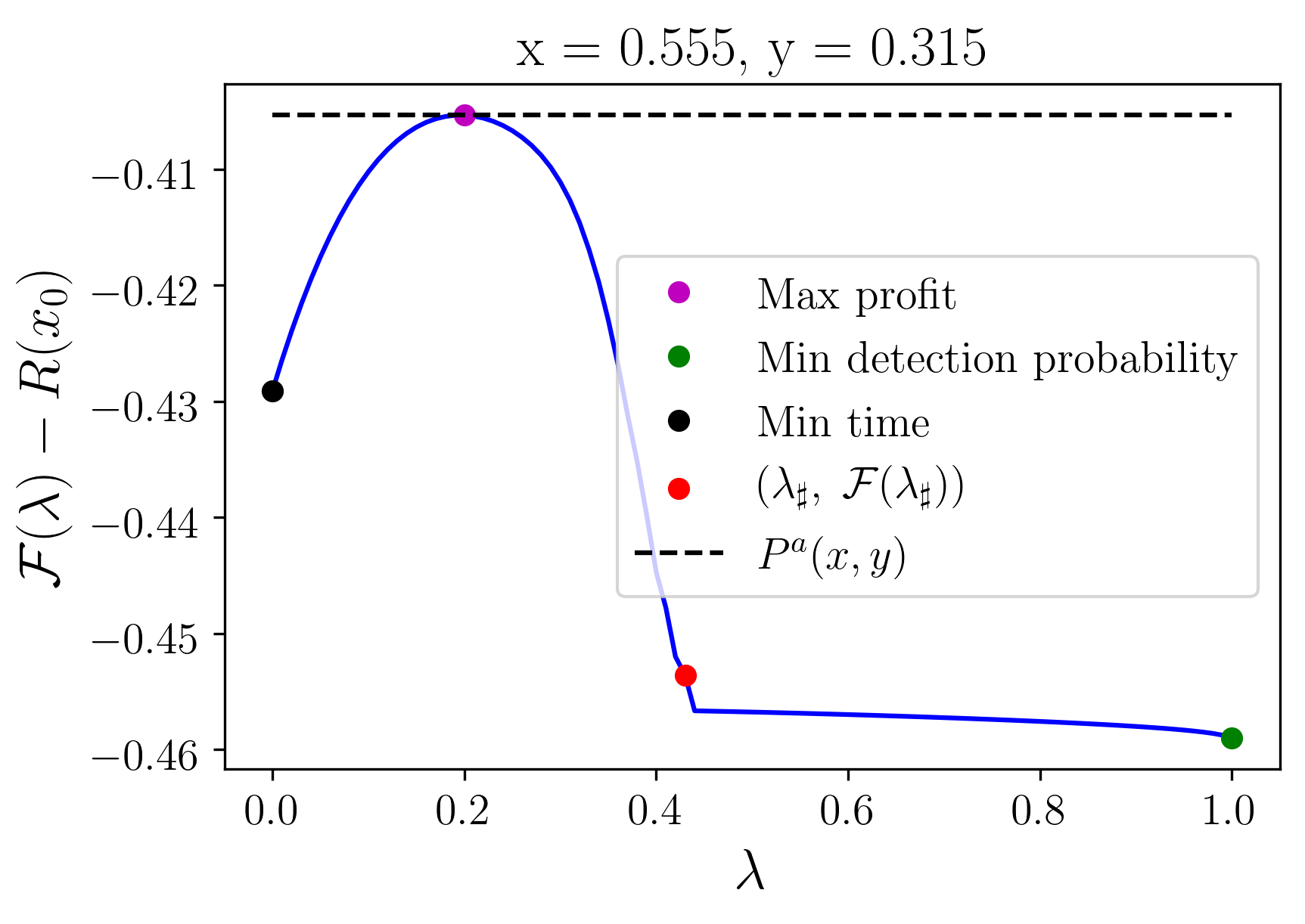}
	\subcaption{\centering $\lambda$-optimization problem}\label{fig:lambda_optimization}
\end{subfigure} 
\begin{subfigure}[t]{0.48\textwidth}
	\includegraphics[height=0.7\textwidth]{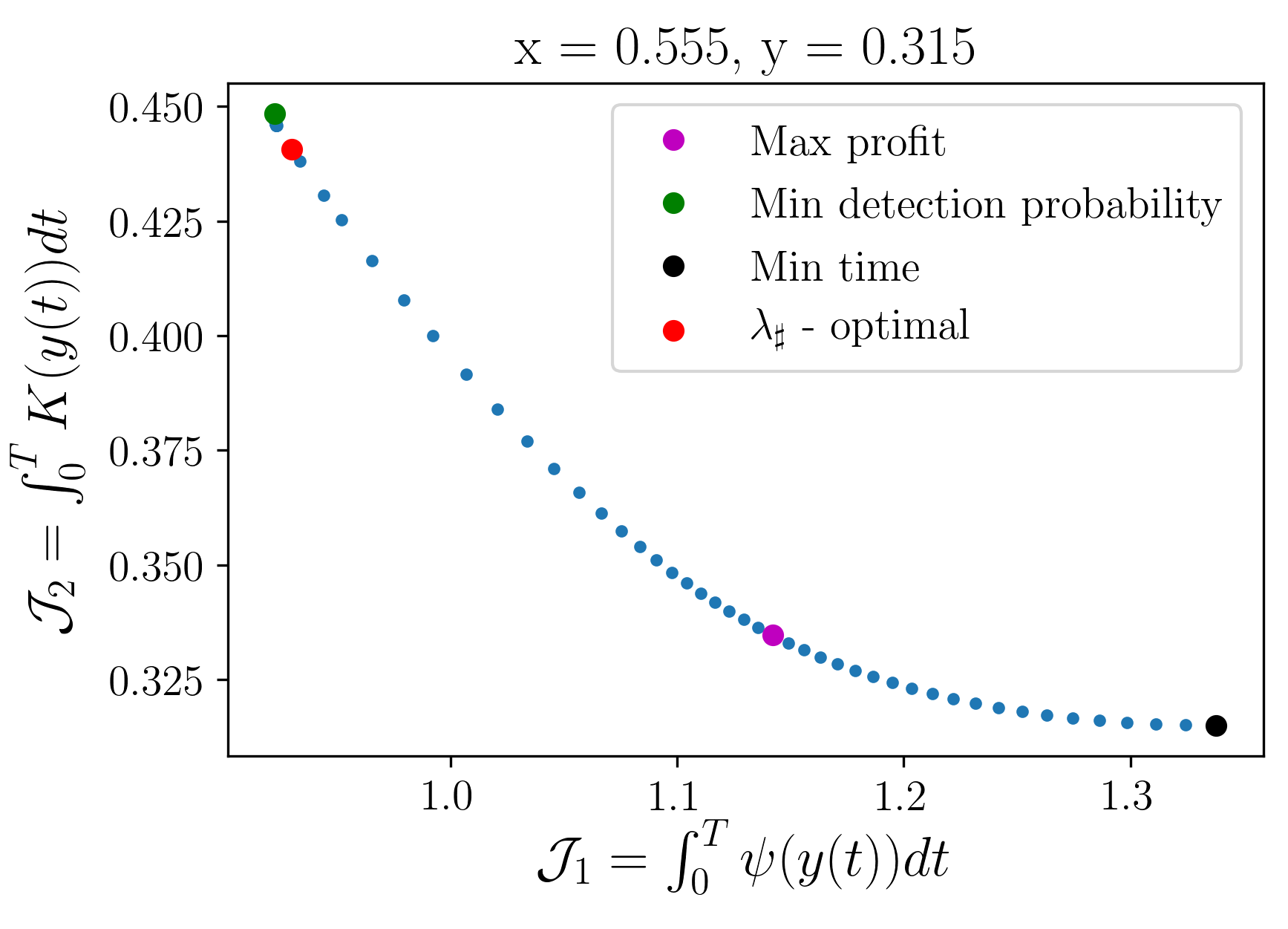}
	\subcaption{\centering $\left(\J_1,\J_2\right)$ Pareto Front}\label{fig:PF}
\end{subfigure} 
\caption{Lambda Optimization and Pareto Front at $\x_0 = (0.555, 0.315)$ in Example 4.}
\label{fig:lambda}
\end{center}
\end{figure}

\subsection{Example 5: exploiting patrol gaps}
The main purpose of this example is to highlight the differences in post-extraction trajectories
between Models A and G.
The domain $\Omega$ is again the unit square $[0,1] \times [0,1]$, with a uniform speed of $f=1$. 
The benefit function $B$ is
\begin{equation*}
B(x,y) = 3 + 7.5\exp\left[-10\left((x-0.7)^2 + (y-0.5)^2\right)\right]
\end{equation*}
and can be seen in Figure \ref{fig:ex5_B}.
The patrol density $\psi(x,y)$ is a sum of eight Gaussians with different means and standard deviations and can be seen in Figure \ref{fig:ex5_psi}.

In Figure \ref{fig:ex5_PA} we see the expected profit $\PA$ from Model A computed on a uniform rectangular grid with $N_x = N_y = 501$, and $N_{\lambda}=401$. 
The optimal post-extraction trajectory corresponding to 
$\hat{\x} = (0.824, \, 0.58)$ is shown in magenta, and $\pBoundaryA$ is shown in orange.

In Figure \ref{fig:ex5_PG} we see the expected profit $\PG$ from Model G computed on the same grid, and $\pBoundaryG$ is shown in orange.
While $\PG$ is only slightly larger than $\PA$, a significant difference can be seen in the post-extraction trajectories.
The solid magenta line represents the optimal post-extraction (pre-detection) trajectory starting from $\hat{\x}.$
The dotted magenta lines show possible post-detection trajectories from three different positions.
Since $f=1$, every post-detection trajectory is just a straight line to the closest point on $\boundary$.

\begin{figure}
\begin{center}
\begin{subfigure}[t]{0.52\textwidth}
	\includegraphics[height=0.65\textwidth]{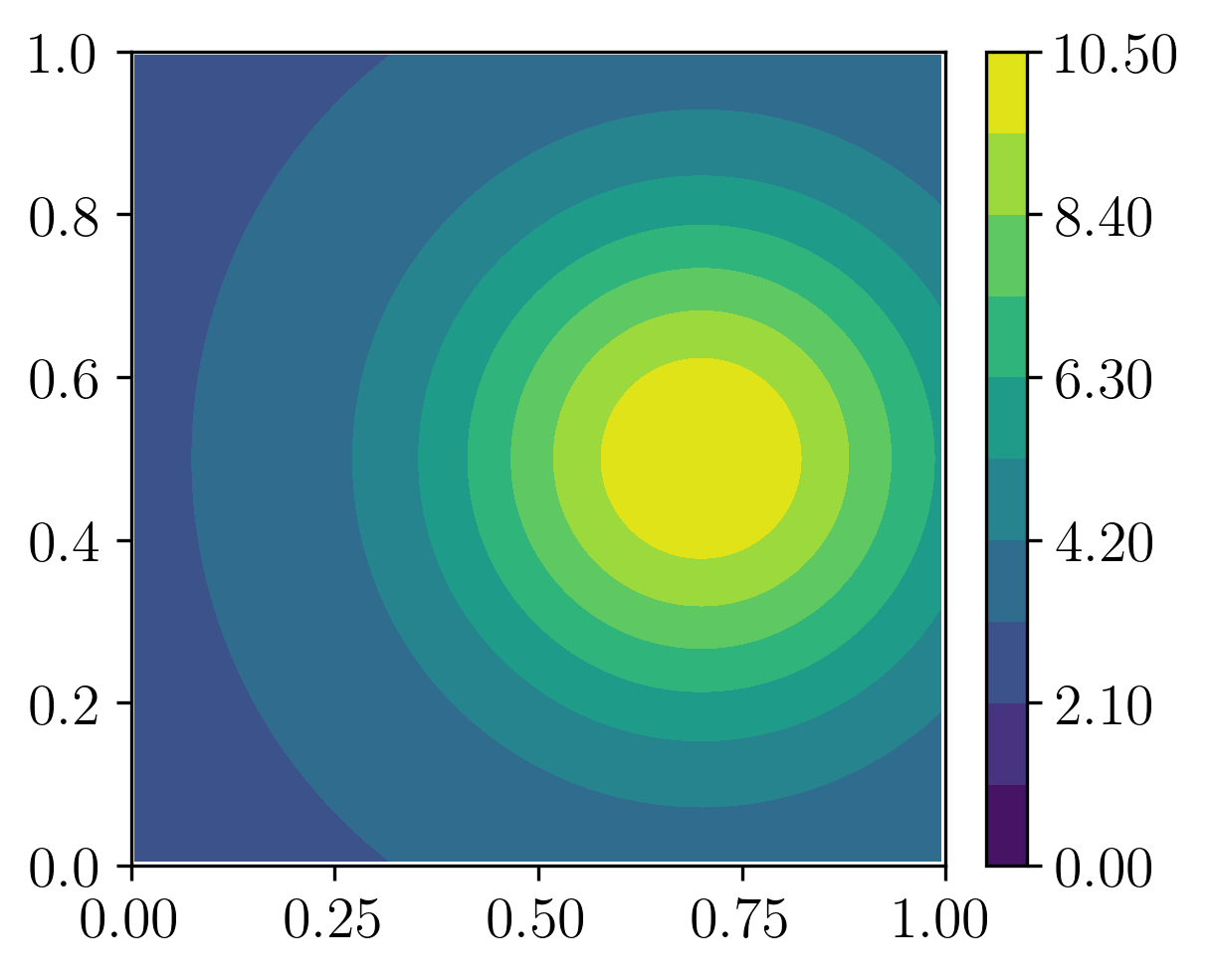}
	\subcaption{Benefit function $B(x,y)$}\label{fig:ex5_B}
\end{subfigure}
\hspace{-0.4in}
\begin{subfigure}[t]{0.52\textwidth}
	\includegraphics[height=0.65\textwidth]{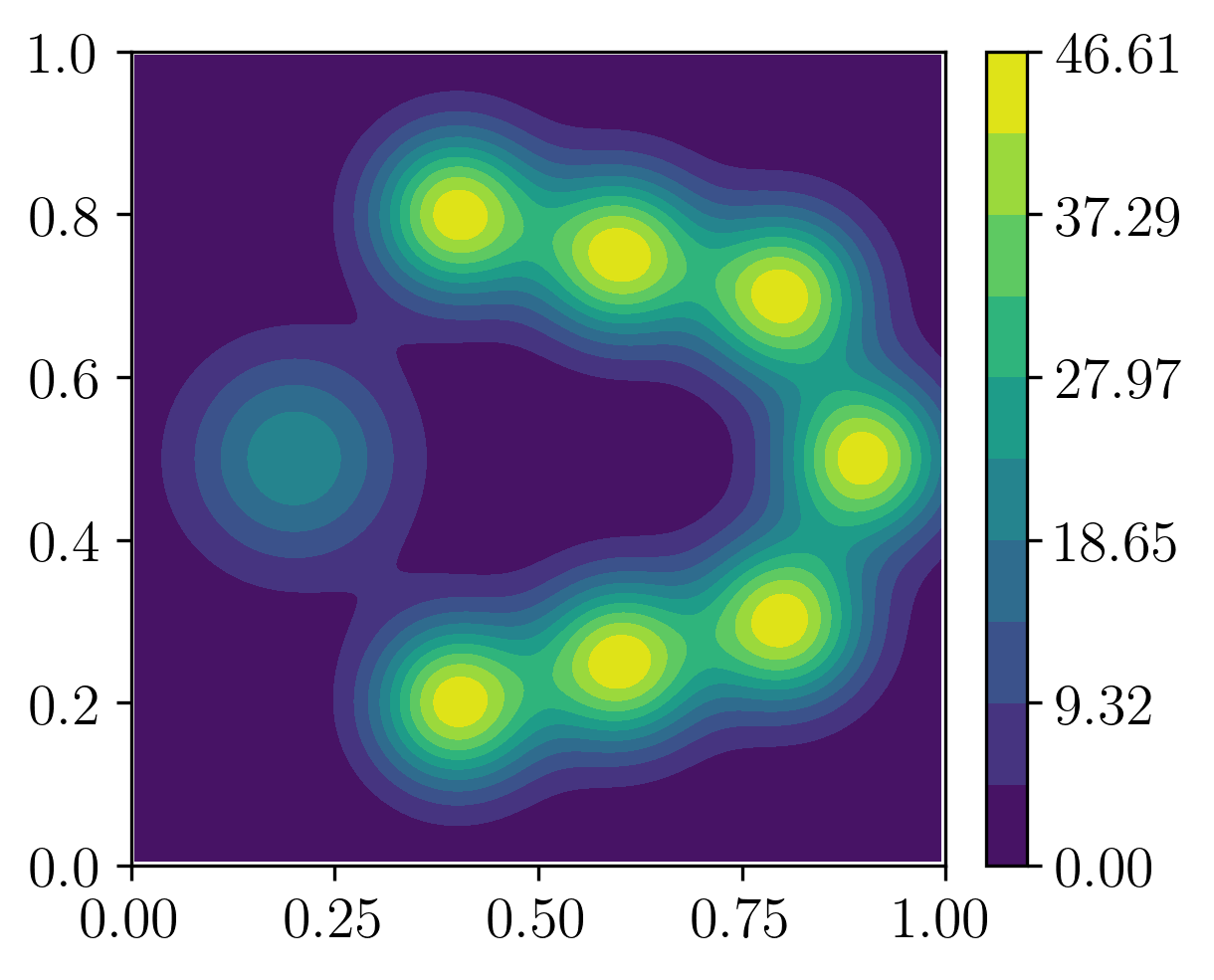}
	\subcaption{Banded patrol density function $\psi(x,y)$}\label{fig:ex5_psi}
\end{subfigure} 

\begin{subfigure}[t]{0.52\textwidth}
	\includegraphics[height=0.65\textwidth]{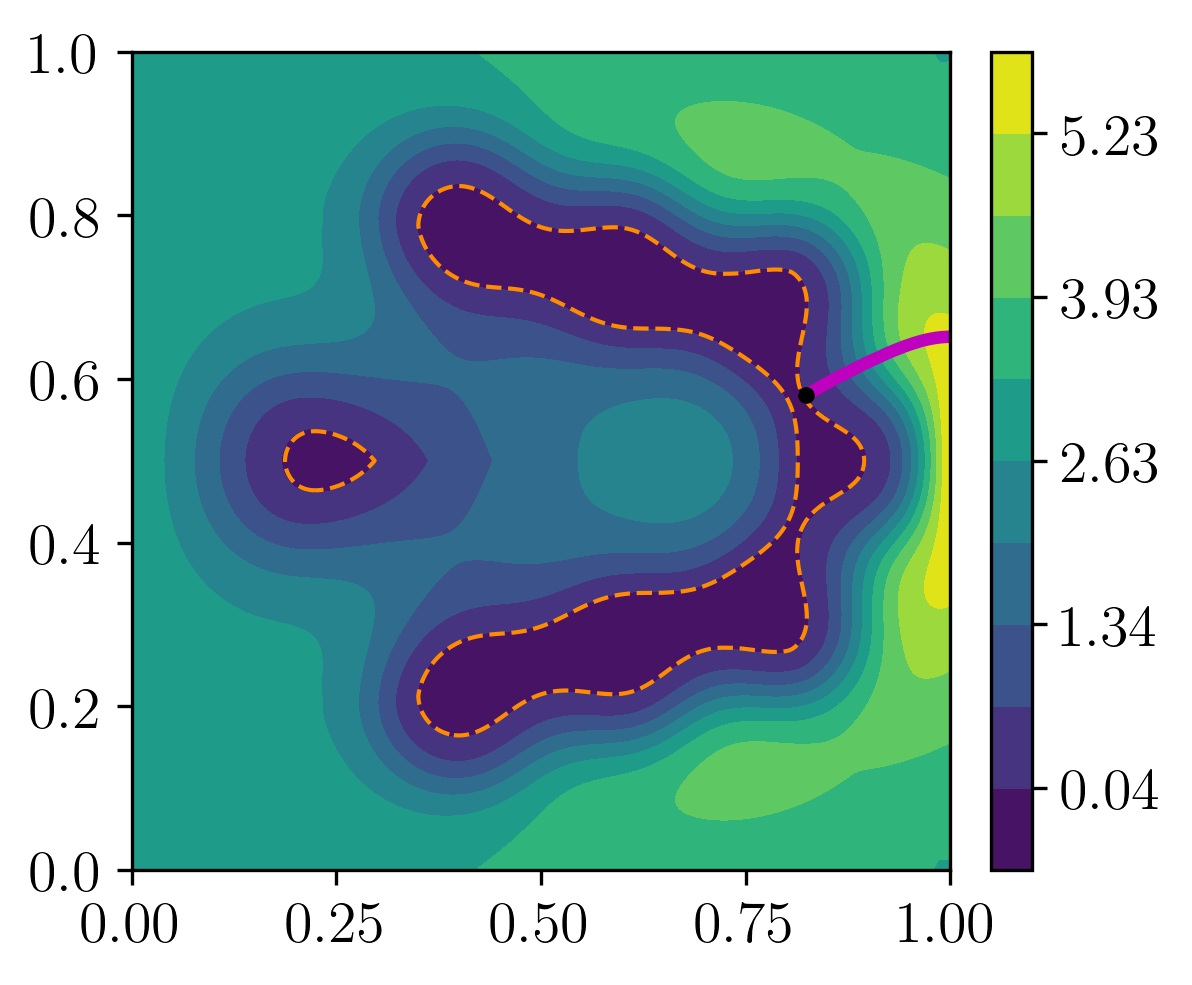}
	\subcaption{\centering Expected profit $\PA(x,y)$}\label{fig:ex5_PA}
\end{subfigure}
\hspace{-0.4in}
\begin{subfigure}[t]{0.52\textwidth}
	\includegraphics[height=0.65\textwidth]{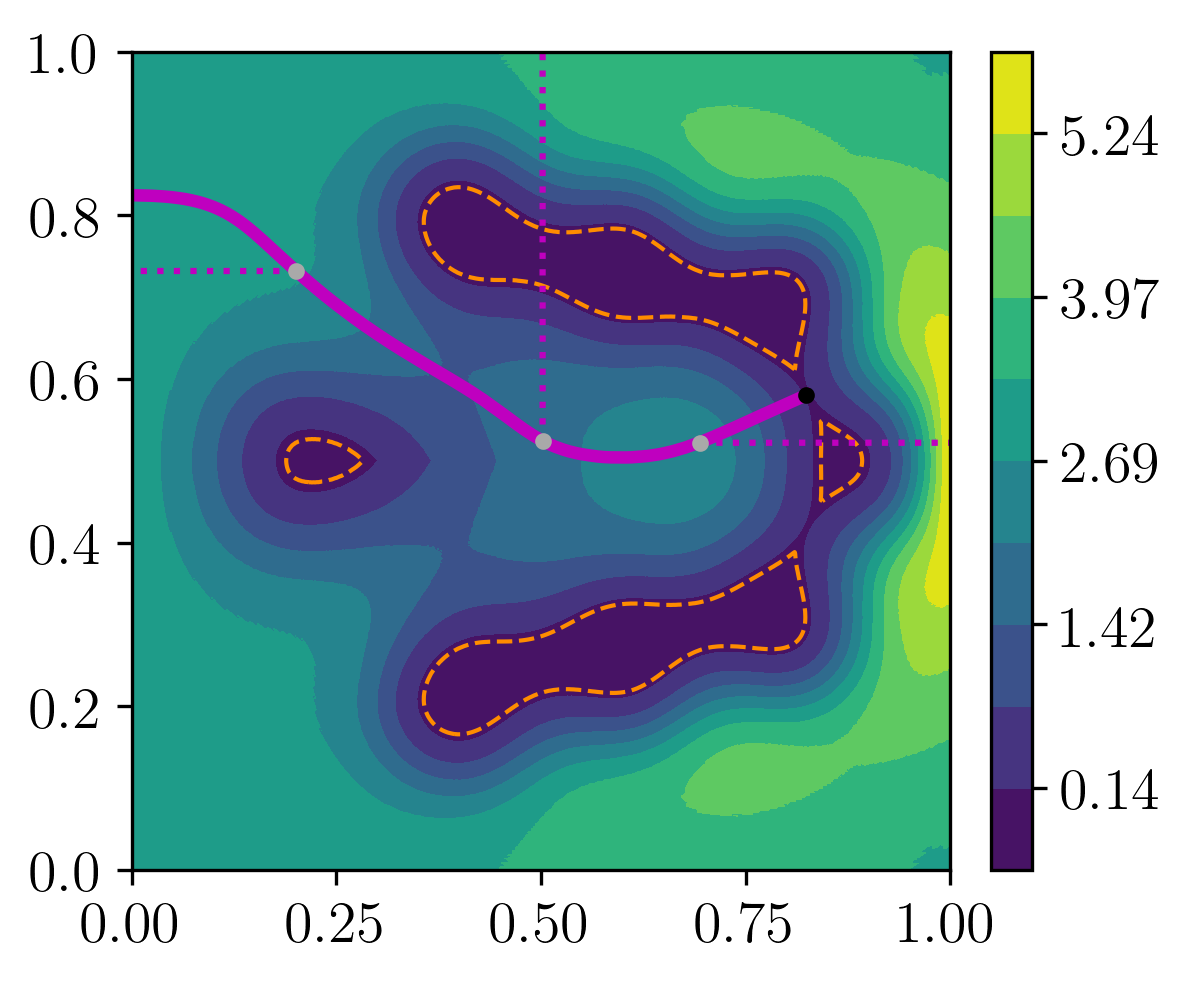}
	\subcaption{Expected profit $\PG(x,y)$}\label{fig:ex5_PG}
\end{subfigure}
\caption{Differences between path-planning in Model A ($\PbarA = 5.87, \,\ApA = 11.06\%, \, \VpA = 14.66\%.$) and Model G ($\PbarG = 5.88, \, \ApG = 8.63\%, \, \VpG = 11.02\%$).}\label{fig:ex5}
\end{center}
\end{figure}

\subsection{Example 6: Yosemite National Park}
Here we use the domain and terrain data from Yosemite National Park.  To facilitate the comparison, we use the same $B,$ $\psi,$ and $f$ defined in \cite{arnold2018modeling}. 

The benefit function shown in Figure \ref{fig:ex6_B} is $B(d) = 8d(2d_m-d)/(d_m),$ 
where $d_m = \max_{\x \in \domain} d(\x).$  
The banded patrol density shown in Figure \ref{fig:ex6_psi} is defined by
\begin{equation*}
\psi(d) = \begin{cases}
\frac{\mu\left(0.7d_m - d\right)}{d_m}, \qquad &0.3d_m < d < 0.7d_m \\
0, \qquad &\text{otherwise}
\end{cases}
\end{equation*}
where $\mu$ is selected based on $E = 3 \times 10^4$.

The extractors' isotropic speed of motion is shown in Figure \ref{fig:ex6_f} and depends on the slope of the elevation map $z(\x)$.  Namely, $s = |\nabla z(\x)|$ is the grade in the direction of $\nabla z(\x)$ and 
\begin{equation}
\label{eq:speed_grade_UCLA}
f(s) = 1.11\exp\left(-\frac{(100s+2)^2}{2345}\right).
\end{equation}

This example was computed on a uniform rectangular grid with $N_x=5406$, $N_y = 4325$, and $N_{\lambda}=21$.
The expected profit $\PA$ can be seen in Figure \ref{fig:ex6_P}.  
The red dashed line corresponds to the boundary of the ``high value region'' 
\begin{equation}
\label{eq:high_value_modified}
\domain_e \; = \; 
\left\{ 
\xbar  \in \domain \, \mid \, \PA(\xbar) + R(\xbar) \, \geq  \, (1-\varepsilon) \max\limits_{\x \in \domain} \left( \PA(\x) + R(\x) \right) 
\right\},
\end{equation} 
with $\varepsilon = 0.85.$  This definition is chosen for the sake of direct comparison with Figure 9 in \cite{arnold2018modeling} since that paper does not consider the cost of pre-extraction trajectories and their definition of pristine region is also different.
\begin{figure}
\begin{center}
\begin{subfigure}[t]{0.52\textwidth}
	\includegraphics[height=0.65\textwidth]{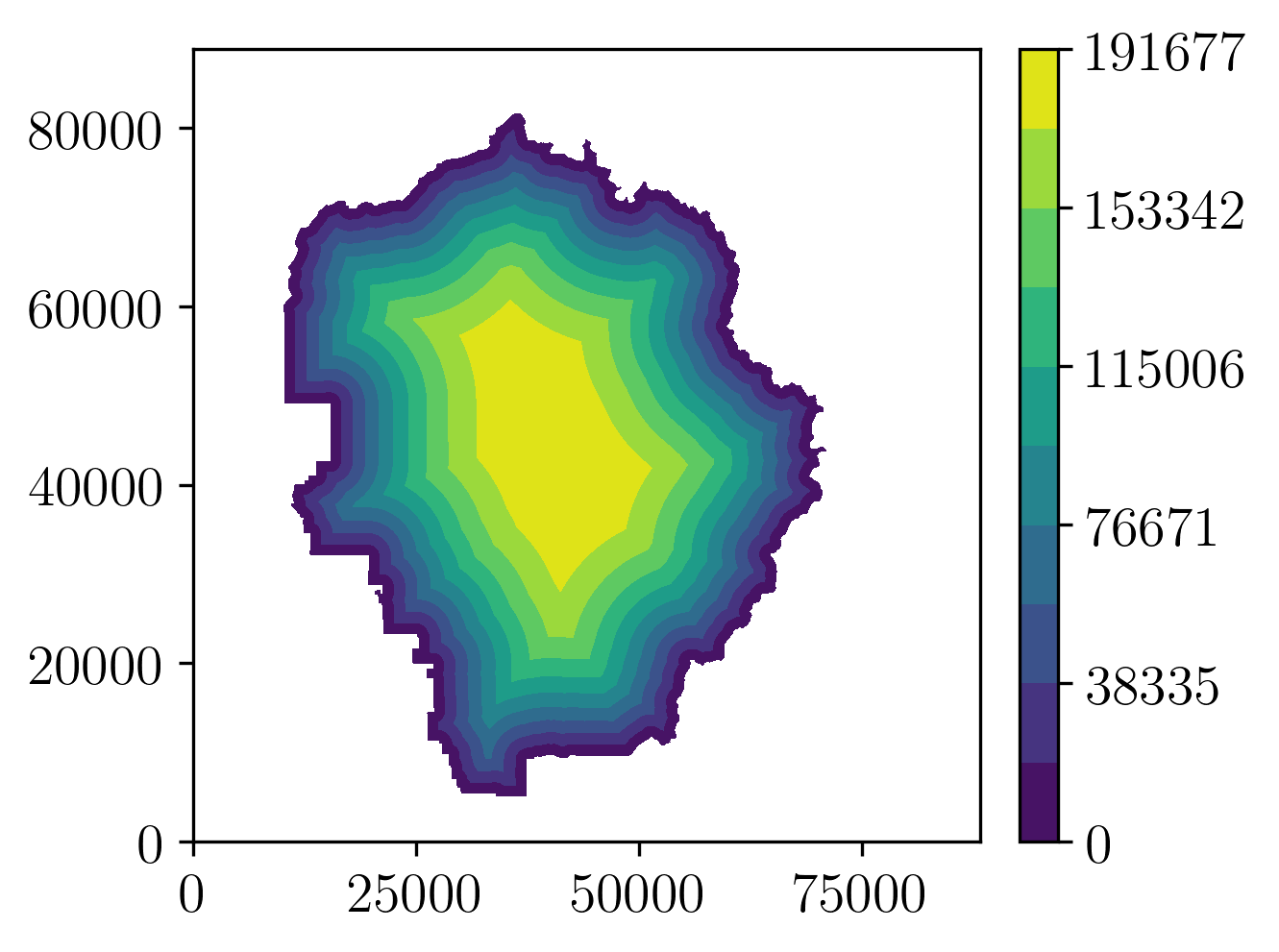}
	\subcaption{\centering Benefit function $B(x,y)$}\label{fig:ex6_B}
\end{subfigure}
\hspace{-0.4in}
\begin{subfigure}[t]{0.52\textwidth}
	\includegraphics[height=0.65\textwidth]{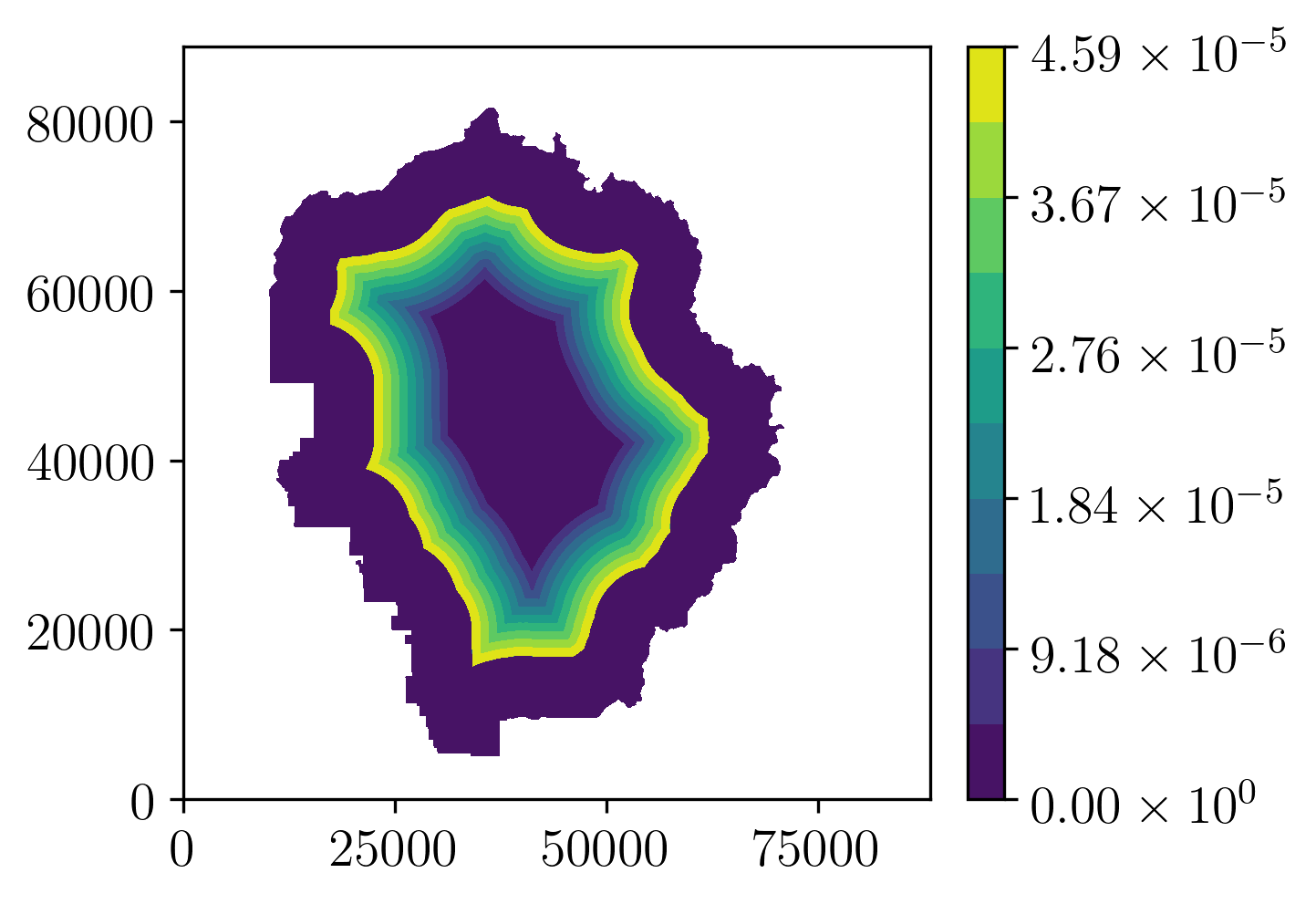}
	\subcaption{\centering Banded patrol density function $\psi(x,y)$}\label{fig:ex6_psi}
\end{subfigure} 

\begin{subfigure}[t]{0.52\textwidth}
	\includegraphics[height=0.65\textwidth]{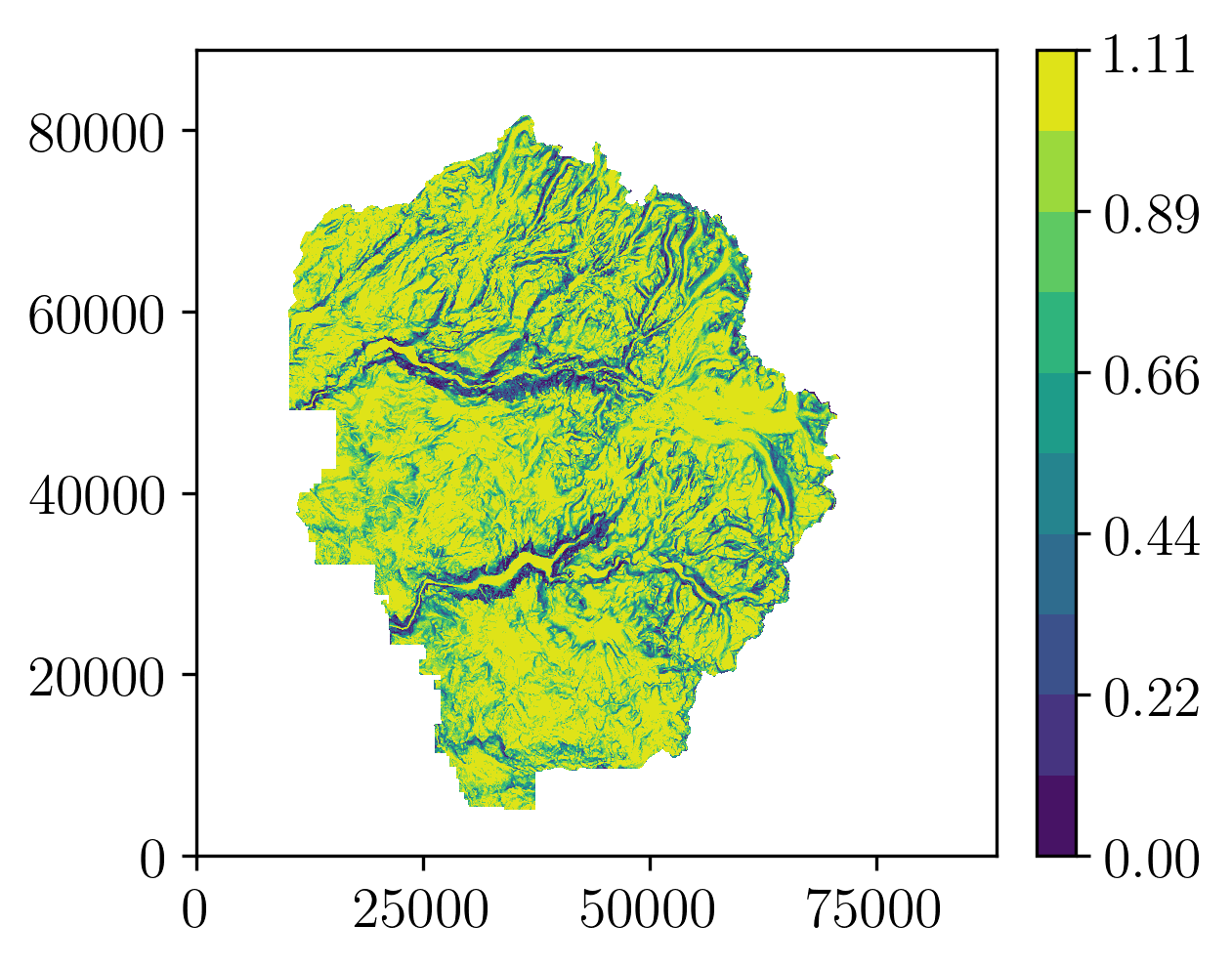}
	\subcaption{\centering Speed $f(x,y)$}\label{fig:ex6_f}
\end{subfigure}
\hspace{-0.4in}
\begin{subfigure}[t]{0.52\textwidth}
	\includegraphics[height=0.65\textwidth]{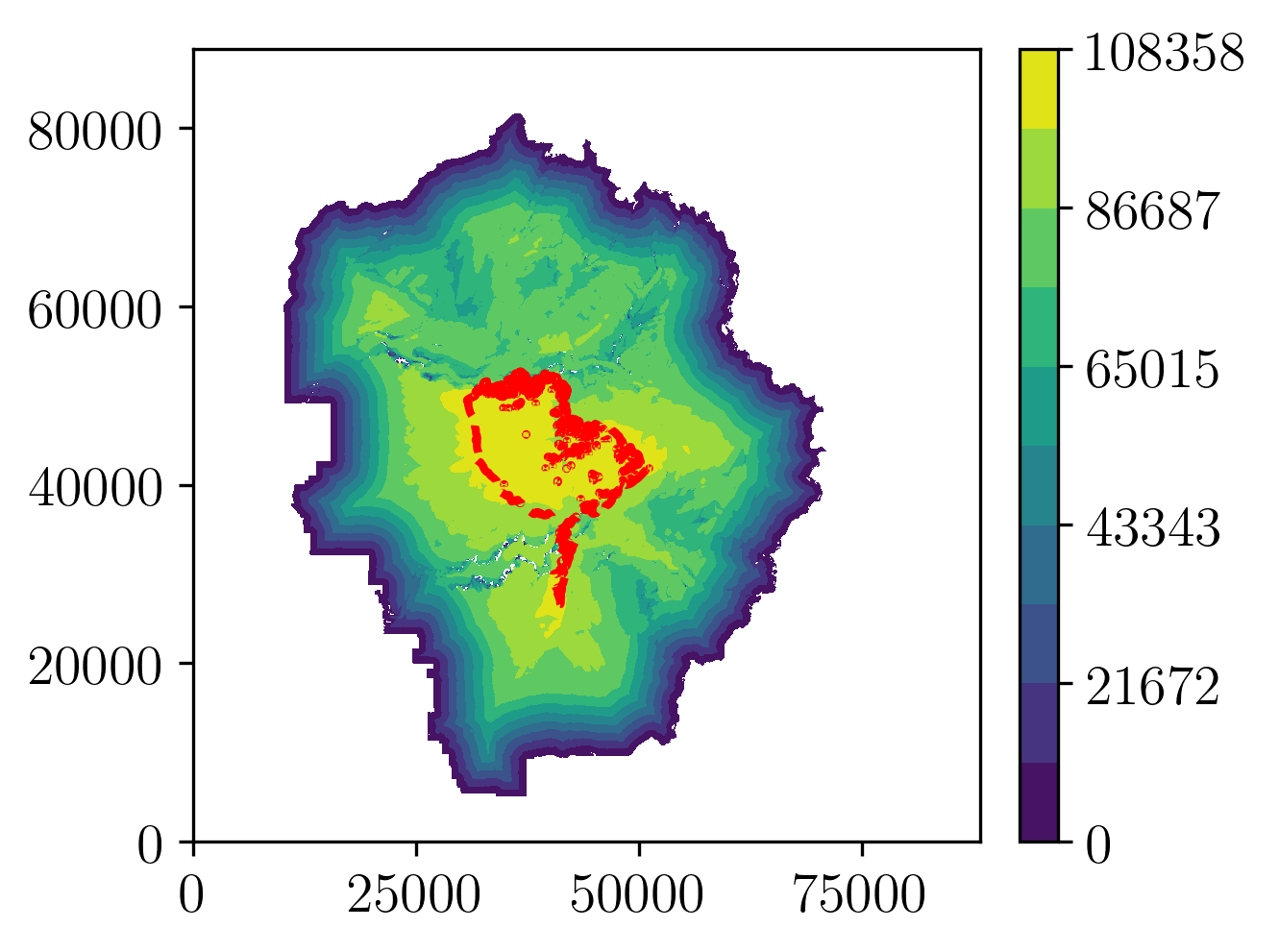}
	\subcaption{\centering Expected profit $\PA(x,y)$ for the extractor}\label{fig:ex6_P}
\end{subfigure}
\caption{Yosemite National Park with $E=3 \times 10^4,$ corresponding to Figure 9b
in \cite{arnold2018modeling}.
$\PbarA = 108358, \, \ApA = 1.90\%,$ and $\VpA = 1.86\%.$
Dashed red line indicates $\boundary_e.$}
\label{fig:ex6}
\end{center}
\end{figure}

Under our interpretation of the extraction criteria, the extractors will target more than $98\%$ of the park, and the patrol budget $E$ is clearly insufficient.  We repeat this experiment with a higher $E =  2 \times 10^5;$ see Figure \ref{fig:ex6_higherE}.  

This example presents a non-trivial Pareto Front since the detection rate $\psi$ is $d$-banded (rather than $\tau$-banded) and the time-optimal paths are generally not detection-minimizing.  As a result, the predicted profit in Model A is different from Model G,
but the differences are not very large here; e.g., $\ApA - \ApG \approx 2.16\%$ even with the higher patrol budget used in Figure \ref{fig:ex6_higherE}.  

It is also worth noting that the definition \eqref{eq:speed_grade_UCLA} results in very small $f$ 
(and occasionally even zero up to machine precision) on all steep parts  of $\domain$.  As a result, the cost of targeting those points becomes extremely high even if everything around them is extracted.  This explains the 
``noisy'' look of $\boundary_e$ in Figure \ref{fig:ex6_P} and $\pBoundary$ in Figure \ref{fig:ex6_higherE}.  Such un-extractable points are not identified in Figure 9b of \cite{arnold2018modeling} since they are masked by the $\Btilde$-interpolation procedure used in that paper.

\begin{figure}
\begin{subfigure}[t]{0.52\textwidth}
	\includegraphics[height=0.65\textwidth]{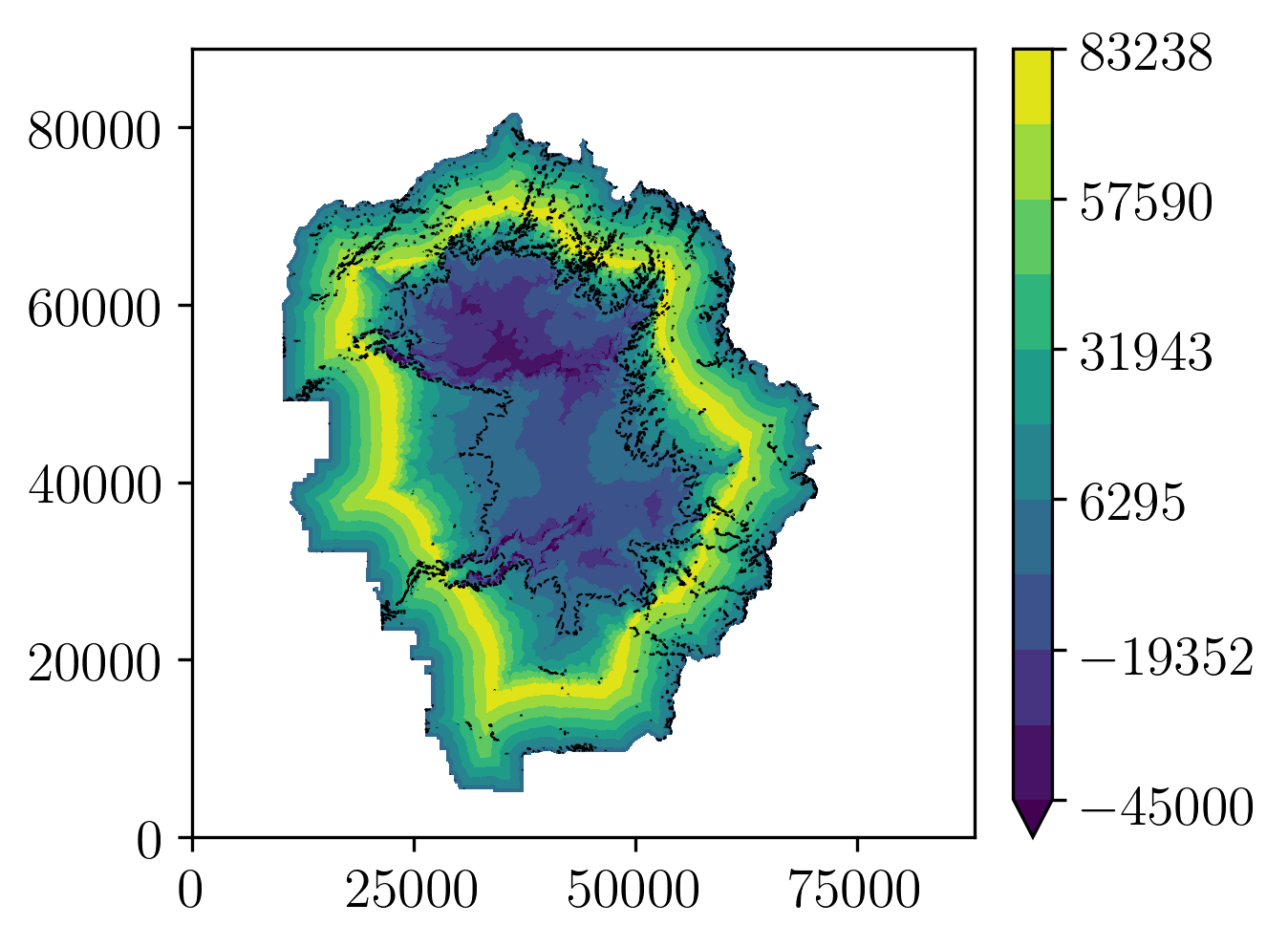}
	\subcaption{\centering Expected profit $\PA(x,y)$\hspace{\textwidth}
	$\ApA = 26.17\%,$ and $\VpA = 43.38\%.$}\label{fig:ex6_PA}
\end{subfigure}
\hspace{-0.4in}
\begin{subfigure}[t]{0.52\textwidth}
	\includegraphics[height=0.65\textwidth]{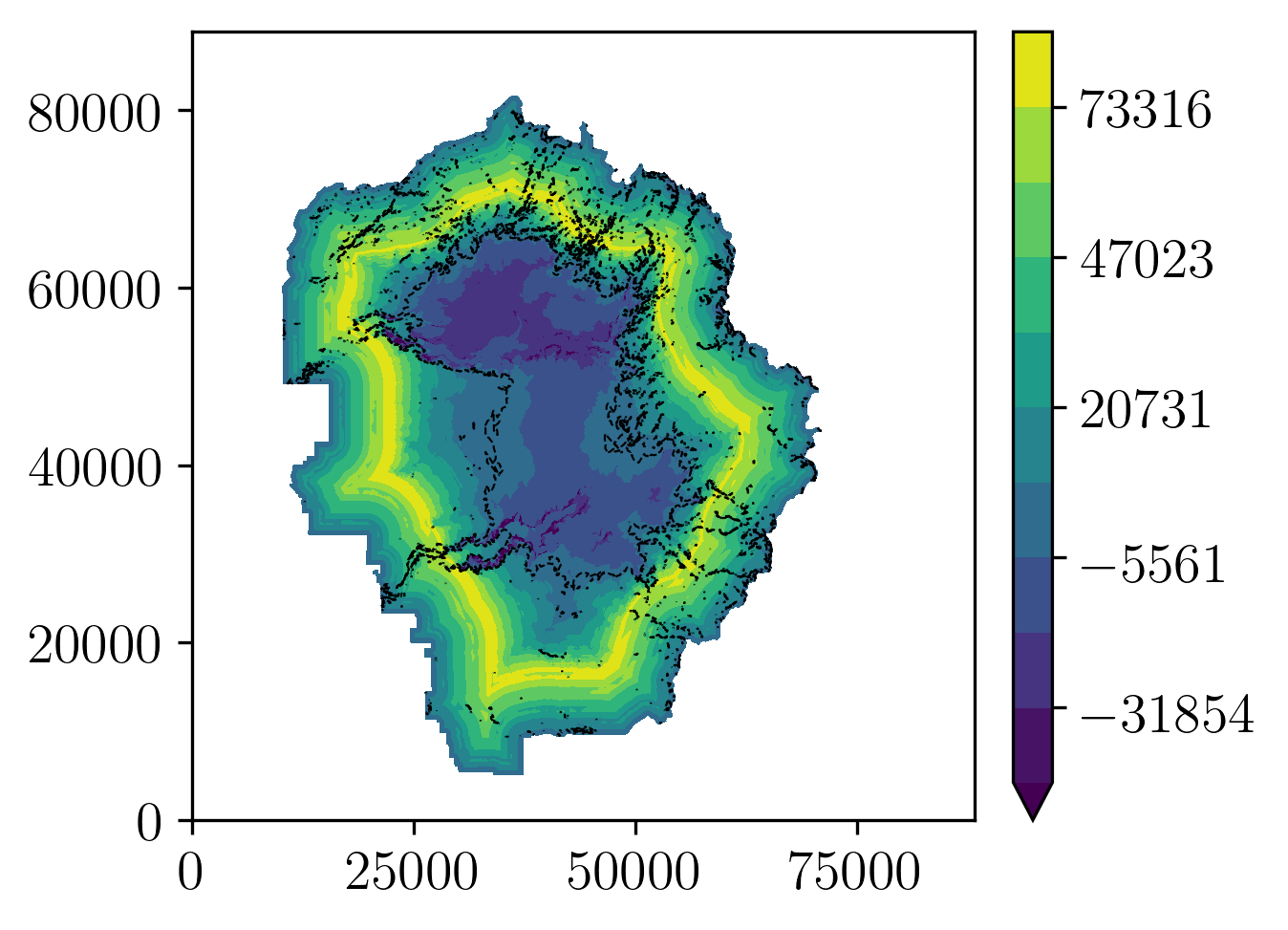}
	\subcaption{\centering Expected profit $\PG(x,y)$\hspace{\textwidth}
	$\, \ApG = 24.01\%,$  and $\VpG= 40.13\%$.}\label{fig:ex6_PG}
\end{subfigure}%
\caption{Yosemite National Park with $E = 2 \times 10^5.$
Black line indicates $\pBoundary.$}
\label{fig:ex6_higherE}
\end{figure}

\subsection{Example 7: Kangaroo Island}
We use the domain and terrain data corresponding to Kangaroo Island with the same $B, \psi,$ and $f$ considered in \cite{arnold2018modeling}.

The functions $B$ and $f$ are defined exactly as in Example 6; see Figures \ref{fig:ex7_B} and \ref{fig:ex7_f}.  The detection rate $\psi$ is a positive constant from $x=30,000$ to $x=90,000$ and zero elsewhere, scaled to match the budget $E = 3 \times 10^4;$ see Figure \ref{fig:ex7_psi}.

This example was computed on a uniform rectangular grid with $N_x=7159$, $N_y = 3111$, and $N_{\lambda}=21$.
The expected profit $\PA$ can be seen in Figure \ref{fig:ex7_P}.
The red dashed line corresponds to the boundary of the ``high value region'' defined in \eqref{eq:high_value_modified} and can be compared with Figure 12b in \cite{arnold2018modeling}. 
Again, almost the entire $\domain$ is targeted by the extractors.  
Since the terrain is mostly flat, this remains true even with a much higher budget $E= 2.5 \times 10^5$; see Figure \ref{fig:ex7_PA}.  

In Figure \ref{fig:ex7_PG} we examine Model G with the same parameters. 
Since a detection is much less costly in Model G, the extractors are even more willing to enter the patrolled area, and the pristine area (shown in orange) is about two times smaller than in Figure \ref{fig:ex7_PA}.

\begin{figure}
\begin{subfigure}[t]{0.52\textwidth}
	\includegraphics[height=0.4\textwidth]{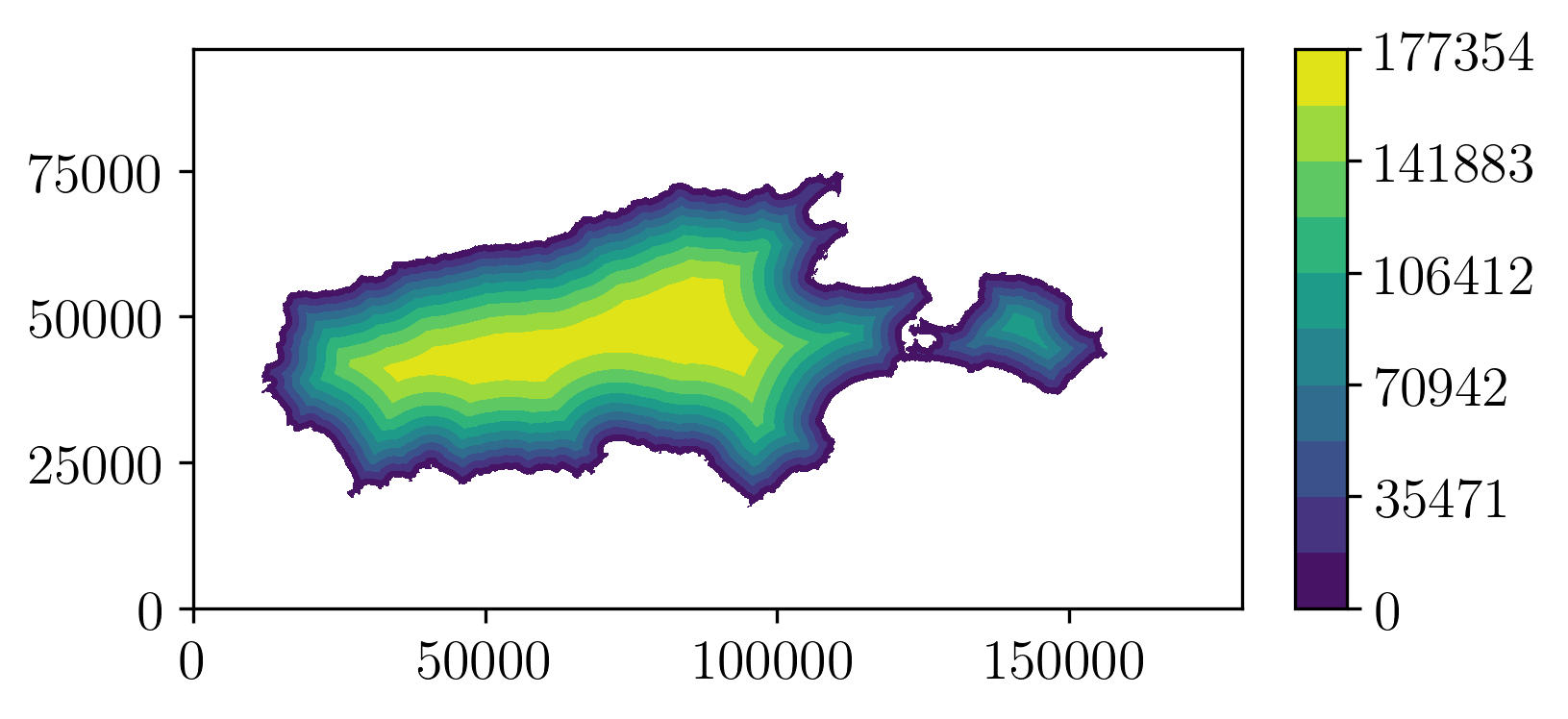}
	\subcaption{\centering Benefit function $B(x,y)$}\label{fig:ex7_B}
\end{subfigure}
\hspace{-0.4in}
\begin{subfigure}[t]{0.52\textwidth}
	\includegraphics[height=0.4\textwidth]{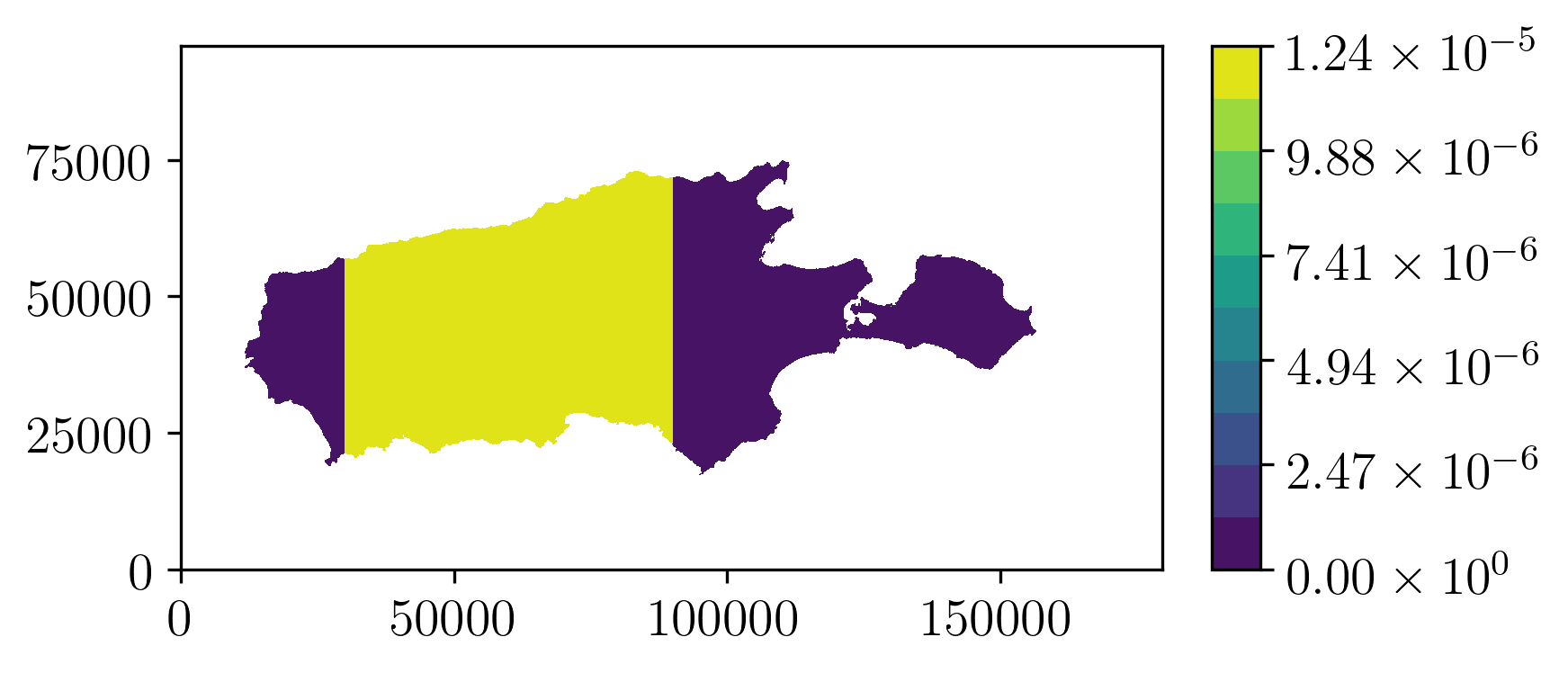}
	\subcaption{\centering Piecewise constant patrol density function $\psi(x,y)$}\label{fig:ex7_psi}
\end{subfigure} 

\begin{subfigure}[t]{0.52\textwidth}
	\includegraphics[height=0.4\textwidth]{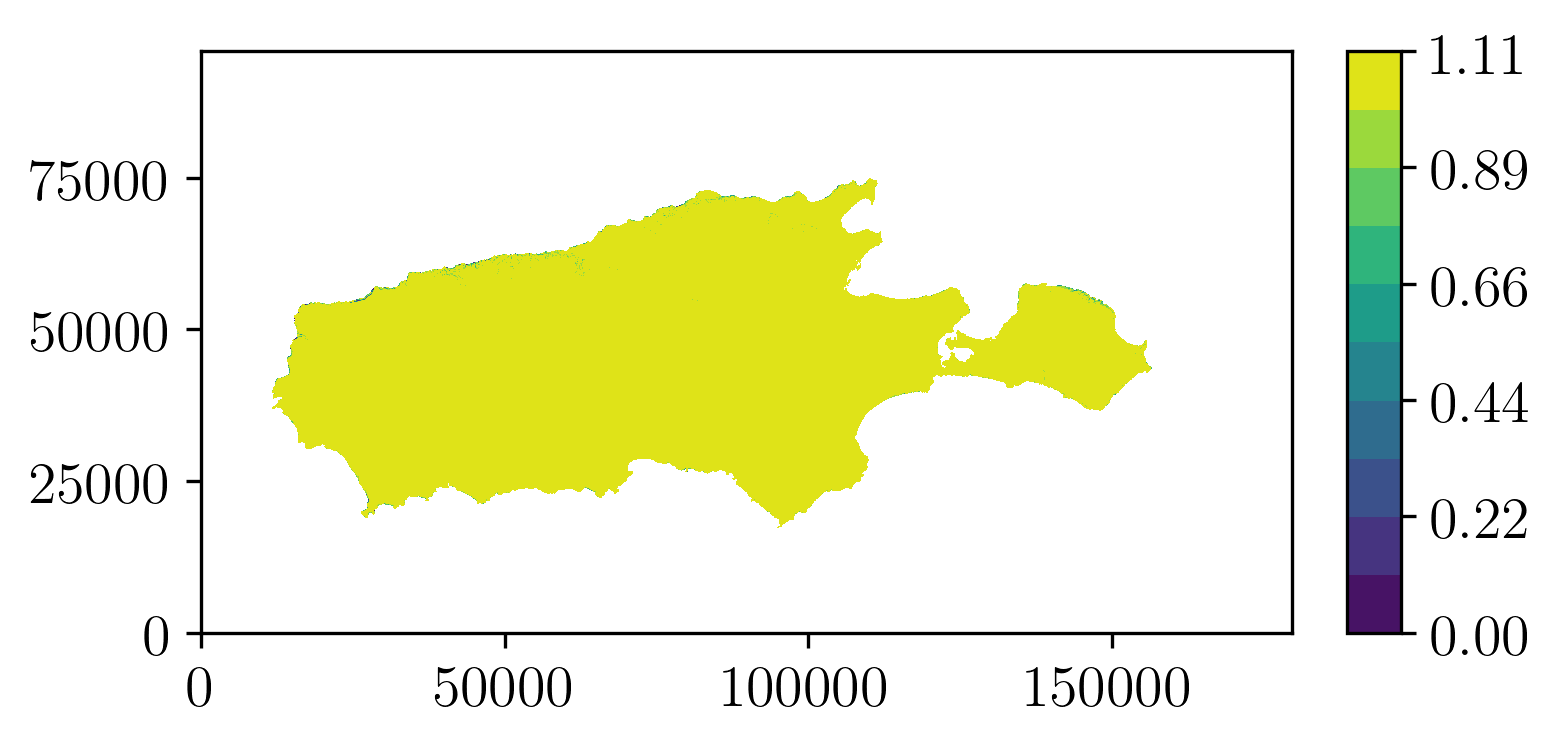}
	\subcaption{\centering Speed $f(x,y)$}\label{fig:ex7_f}
\end{subfigure}
\hspace{-0.4in}
\begin{subfigure}[t]{0.52\textwidth}
	\includegraphics[height=0.4\textwidth]{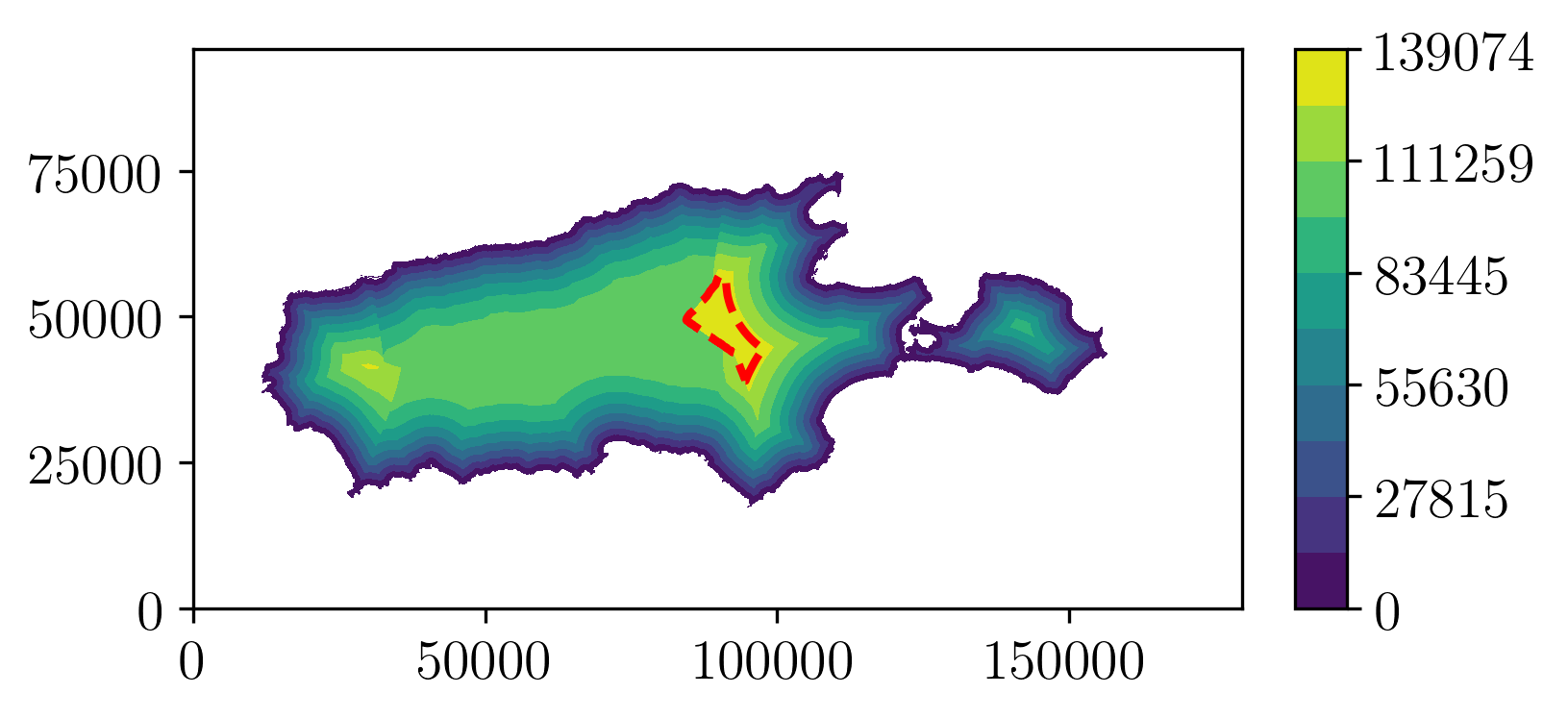}
	\subcaption{\centering Expected profit $\PA(x,y)$ for the extractor}\label{fig:ex7_P}
\end{subfigure}
\caption{Kangaroo Island with $E=3 \times 10^4;$ corresponding to Figure 9b
in \cite{arnold2018modeling}. $\, \PbarA = 139074, \ApA = 0.16\%,$ and $\VpA =0.0042\%.$
Dashed red line indicates $\boundary_e.$}\label{fig:ex7}
\end{figure}

\begin{figure}
\begin{subfigure}[t]{0.52\textwidth}
	\includegraphics[height=0.4\textwidth]{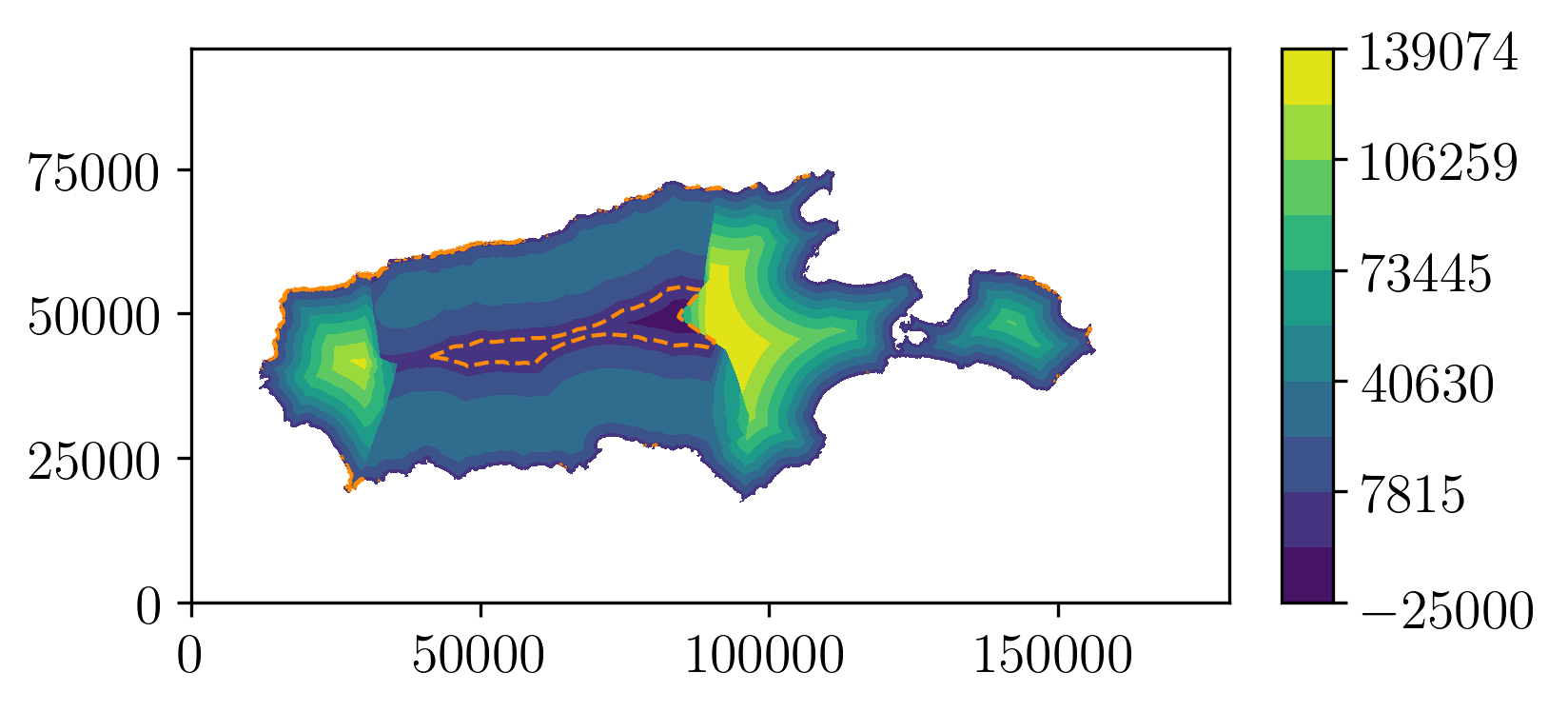}
	\subcaption{\centering Expected profit $\PA(x,y)$\hspace{\textwidth}
	$\, \ApA = 4.42\%,$  and $\VpA = 8.29\%$.}\label{fig:ex7_PA}
\end{subfigure}
\hspace{-0.4in}
\begin{subfigure}[t]{0.52\textwidth}
	\includegraphics[height=0.4\textwidth]{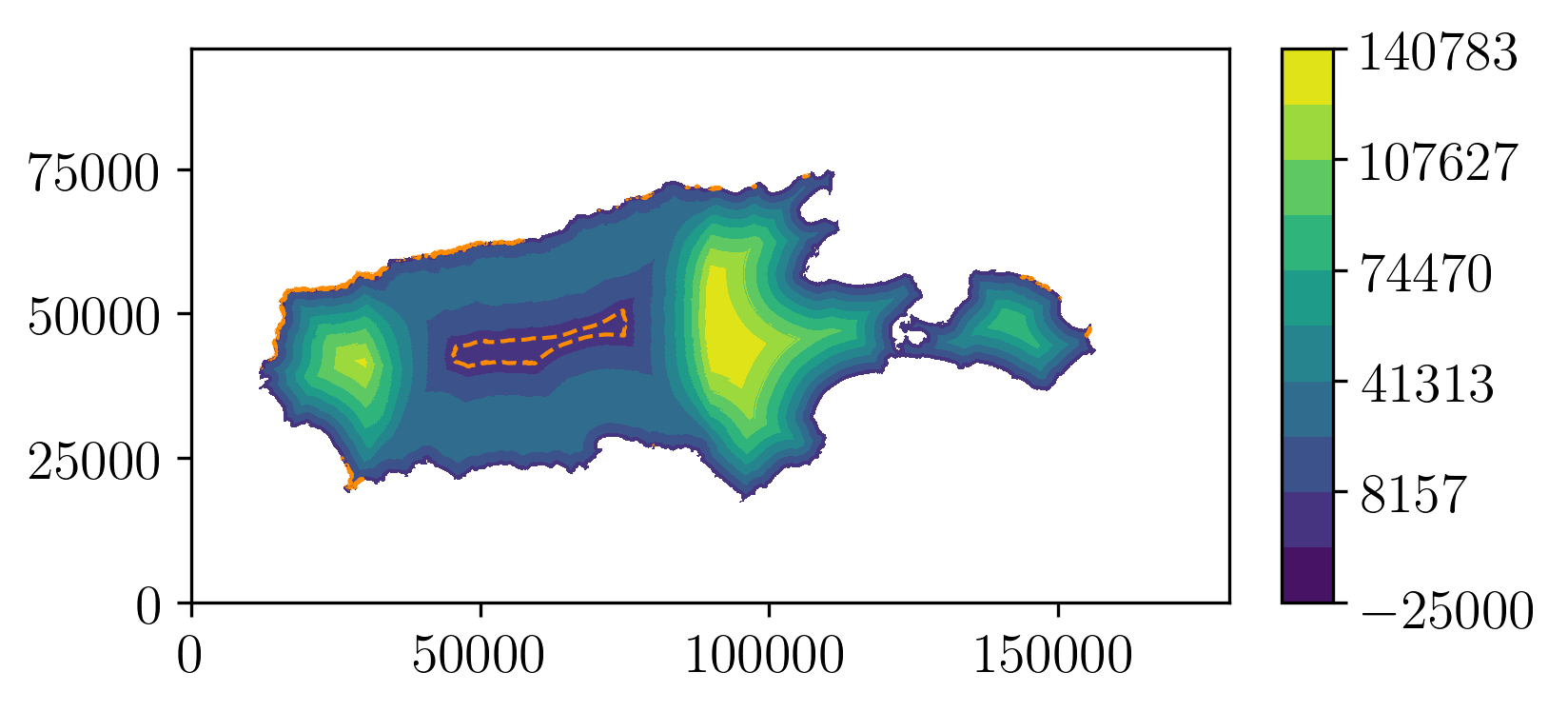}
	\subcaption{\centering Expected profit $\PG(x,y)$\hspace{\textwidth}
	$\, \ApG = 2.19\%,$  and $\VpG= 3.90\%$.}\label{fig:ex7_PG}
\end{subfigure}%
\caption{Kangaroo Island with $E= 2.5 \times 10^5.$  Dashed orange line indicates $\pBoundary.$}%
\end{figure}

\section{Conclusions}
\label{s:Conclusions}

We have presented two control-theoretic models of illegal resource extraction 
counteracting the detection efforts of Protected Area Managers (PAMs).
The perpetrators are assumed to 
choose
potential extraction sites based on their expected profit. 
They take into account the distribution of resources, the geometry and terrain of the protected area $\domain$, and the localized detection rate $\psi: \domain \to \R_{+,0}$ due to PAMs' 
patrol activities.  The cost of extracting at $\x \in \domain$ is based on a roundtrip $\boundary \rightarrow \x \rightarrow \boundary.$  The two portions of the trip usually follow different paths.  The $\boundary \rightarrow \x$ path is selected to minimize the time and hardship of travel, while on the way back ($\x \rightarrow \boundary$) these are balanced against the expected losses resulting from a possible detection by PAMs. 
 
A detection is always followed by confiscation of resources carried by the extractor.  
But the details of how and when a confiscation happens significantly influence the perpetrators' decision making process.  
Our ``Model G'' assumes that the extractors are spotted by ground patrols with an immediate capture/confiscation after their detection.
It is further assumed that after the confiscation the extractors switch to the quickest path to leave the protected area.
We 
find 
their best pre-detection trajectories by recasting this as an optimal control of a suitable randomly-terminated process.  
Our ``Model A'' assumes that the extractors are spotted by aerial patrols but remain oblivious when this happens, continuing along their originally chosen trajectory until they are captured at $\boundary.$
The key challenge of this application is the impossibility of writing a single running cost to be integrated along the post-extraction path unless we resort to a linearizing approximation for the probability of detection (see Remark \ref{rem:linearize}) or agree to increase the dimension of our planning space (see Remark \ref{rem:higher_dim}).   We show that the former often leads to very inaccurate predictions while the latter leads to expensive numerical methods.
Instead, we opt to use the techniques of {\em multiobjective} dynamic programming on the original domain $\domain$, resulting in an accurate and efficient numerical implementation.  Our method is illustrated on a range of examples, with two of them based on real terrain maps from Yosemite National Park in California and Kangaroo Island in South Australia.  In each case, we compute the perpetrators' expected profit from all possible extraction sites as well as the ``pristine'' region $\domain_p \subset \domain$ not affected by extraction.  
We also show that Model G generally results in higher expected profits, smaller $\domain_p,$ and different post-extraction trajectories.  

Many of our modeling assumptions were made primarily for the simplicity of exposition and could be easily relaxed in the future.  For example, it would not be hard to incorporate extractor's preferences for starting and terminal positions on $\boundary.$  One could also easily model the time needed to extract at each site and the resulting increase in the risk of detection.  It would be similarly trivial to include additional fines in case of detection/capture.
Following \cite{arnold2018modeling}, we have assumed that the extractors' speed of motion $f$ is isotropic.  I.e., if $z(\x)$ is the elevation map, the speed $f$ was based on the grade in the direction of $\nabla z(\x)$ rather than on the grade in the chosen direction of motion $\ba$.  In reality, a pedestrian's speed should depend on both, particularly for moderate $|\nabla z|.$  Any such model would replace Eikonal PDEs with anisotropic HJB equations.  Efficient numerical methods for the latter have also been developed both in Fast Marching \cite{SethVlad3, alton2012ordered, mirebeau2014efficient}
and Fast Sweeping \cite{tsai2003fast, KaoOsherTsai_2005} frameworks.

While our current implementation is sequential, we note that it would be easy to build a parallel algorithm with excellent scalability for computing $P(\x)$:  
for each specific $\lambda_k$ or $\Btilde_m$, the value function could be computed by a separate CPU core.  
The resulting speed-up will be very useful in optimizing the choice of detection rate $\psi$  to maximize the area of $\domain_p$ while staying within PAMs' enforcement/surveillance budget.  In the current paper we perform some of this 
by straightforward grid searches, but a more general $\psi$-optimization will require higher computational efficiency.  The problem of optimally allocating surveillance resources across several patrol stations (e.g., Example 4) will become particularly expensive as the number of stations grows.  We believe that an efficient solution will be based on generalizing our recent work on surveillance-evasion games \cite{Static_SEG, TD_SEG}.

We conclude by listing two 
significant 
challenges not tackled in the current paper.  Firstly, a ``single period'' assumption for this Stackelberg game is rather limiting.  As extractors continue degrading the protected area $\domain$, the distribution of resources changes and PAMs should accordingly re-allocate the surveillance resources, adjusting $\psi$ to increase the pristine region $\domain_p.$  Secondly, in practice the detection rate $\psi$ is not chosen directly but results from implementing specific surveillance strategies (e.g., choices of particular patrol trajectories and patrol frequencies).   It would be far more realistic and more difficult to optimize those strategies directly.

\vspace*{5mm}
\noindent
{\bf Acknowledgements. }
We are very grateful to all authors of \cite{arnold2018modeling} for answering our questions about their model and for sharing the terrain data used in their numerical experiments.  The current paper was written during AV's sabbatical visit to ORFE/Princeton, and he would like to thank ORFE for its hospitality.  

\bibliographystyle{siam}
\bibliography{Poaching}

\end{document}